\documentclass{article}
\usepackage[utf8]{inputenc}
\usepackage{mathtools}   
\usepackage{amsthm}
\usepackage{amsfonts}
\usepackage{amssymb}
\usepackage{graphicx}
\usepackage{multirow}
\usepackage{epstopdf}
\usepackage{epsfig}
\usepackage{algorithm2e}
\usepackage{tikz}
\usepackage[legalpaper, margin=1.5in]{geometry}
\usepackage{matlab-prettifier}

\usepackage{anyfontsize}

\def\v{{\mathbf v}} \def\y{{\mathbf y}} \def\x{{\mathbf x}}

\def\z{{\mathbf z}} \def\b{{\mathbf b}} \def\e{{\mathbf e}}

 \def\r{{\mathbf r}} \def\q{{\mathbf q}}

\def\s{{\mathbf s}}  \def\g{{\mathbf g}}

\def\w{{\mathbf w}}   \def\p{{\mathbf p}}

\def\c{{\mathbf c}} \def\d{{\mathbf d}} \def\h{{\mathbf h}}

 \def\tilde{\widetilde} \def\hat{\widehat}

\def\tr{\hbox{\,tr\,}}
\def\sd{\hbox{\,sd\,}}
\def\Span{\hbox{\,Span\,}}

{\color{blue} }%
{\color{black} }

\title{A variation of Broyden Class methods using Householder adaptive transforms}

\author{Stefano Cipolla\footnotemark[1] ,  C. Di Fiore\footnotemark[2] , P. Zellini\footnotemark[2]}

\newtheorem{remark}{Remark}
\newtheorem{definition}{Definition}

\newtheorem{corollary}{Corollary}
\newtheorem{theorem}{Theorem}
\newtheorem{lemma}{Lemma}
\newtheorem{assumption}{Assumption}
\newtheorem{prob}{Problem}
\begin{document}
	\maketitle
	\renewcommand{\thefootnote}{\fnsymbol{footnote}}
	\footnotetext[1]{Department of Mathematics ``Tullio Levi-Civita'', University of Padua, Padova, Italy (\texttt{stefano.cipolla@unipd.it}) }
	\footnotetext[2]{ University of Rome ``Tor Vergata'' Via Della Ricerca Scientifica 1, 00133 Rome \\ (\texttt{{difiore@mat.uniroma2.it}}, {\texttt{zellini@mat.uniroma2.it}})}

	\begin{abstract}{In this work we introduce and study  novel Quasi Newton minimization methods based on a Hessian approximation Broyden Class-\textit{type} updating scheme, where a suitable matrix $\tilde{B}_k$ is updated instead of the current Hessian approximation $B_k$. We identify conditions which  imply the convergence of the algorithm and, if exact line search is chosen, its quadratic termination. By a remarkable connection between the projection operation and Krylov spaces, such conditions can be ensured using low complexity matrices $\tilde{B}_k$ obtained projecting $B_k$ onto algebras of matrices diagonalized by
			products of two or three Householder matrices {adaptively chosen step by step}.
			Extended experimental tests show   {that the introduction of the adaptive criterion, which theoretically guarantees the convergence, considerably improves the robustness of the minimization schemes when compared with a non-adaptive choice;} moreover, they show that the proposed methods {could be} particularly suitable to solve large scale problems where $L$-$BFGS$ performs poorly.}

\end{abstract}

\section{Introduction}
In minimizing a function $f:\mathbb{R}^n\rightarrow \mathbb{R}$, in order to reduce the computational cost per iteration and the memory required for implementation of the well known $BFGS$ minimization method, it is proposed in \cite{DFLZ,DFLZHART,DFZ,DFZ1,D} to use a $BFGS$\textit{-type} updating scheme {which updates, at each step, a suitable approximation of the Hessian approximation $ B_{k}$, usually denoted by  $\tilde{B}_{k} $}. This scheme is named $\mathcal{L}$QN when the matrix $\tilde{B}_k$ is the projection $\mathcal{L}_{B_k}$ of the matrix $B_k$ in a matrix algebra {$\mathcal{L}$ of matrices simultaneously diagonalized by a given unitary transform $U$ (we write $\mathcal{L}:=\sd U$, see \eqref{eq:sdu_definition} for a precise definition).} The implementation of the $\mathcal{L}$QN turns out to be very cheap when $U$ defines a low complexity transform.

While in \cite{BDCFZ,CCD,ebrahimib} $\mathcal{L}$ is a fixed matrix algebra, in \cite{D,DFZ} it is observed that an adaptive choice of $\mathcal{L}$, i.e, using different algebras $\mathcal{L}^{(k)}$ for each iteration $k$, could preserve more information from the original matrix $B_k$, and thus improve  the efficiency of $\mathcal{L}$QN. In \cite{CDTZ} it is introduced a convergent $\mathcal{L}^{(k)}$QN scheme whose effectiveness is shown by preliminary numerical experiences.

The main contribution of this work is twofold. On the one hand we extend the theoretical framework and the convergence theory developed in {\cite{DFLZ,CDTZ}} for   $BFGS$-type  techniques to the restricted Broyden Class-type of quasi Newton methods (for the restricted Broyden Class see \cite{BNY}). 

{On the other hand, we consider the special Broyden Class-type methods in which the update of $B_k$ has the form 
	\begin{equation} \label{eq:low_rank_unpdate_introduction}
	B_{k+1}= \Phi(\mathcal{L}^{(k)}_{B_k},\s_k,\y_k,\phi),
	\end{equation}
	where {$\s_k:=\x_{k+1}-\x_k$, $\y_k:=\g_{k+1}-\g_k$ ($\g_{k}:=\nabla f(\x_k)$), $\x_k$ is the current guess of the minimum and } the transform $U_k$, which diagonalizes the matrices of $\mathcal{L}^{(k)}$, is the product of few Householder reflections. Exploiting the fact that a Householder reflection is a rank one modification of the identity, we propose an algorithm to implement the update in equation \eqref{eq:low_rank_unpdate_introduction} using $O(n)$ operations per step: hence the complexity of the Quasi-Newton methods so obtained is comparable to the more traditional methods of limited-memory type.	Additionally, we show that {if} the projections $\mathcal{L}_{B_k}^{(k)}$ {are such that} 
	\begin{equation} \label{eq:numeri_romani_1}
	{\tag{I} 
		\tr{\mathcal{L}^{(k)}_{B_k}} \leq \tr{B_k},\;\;\det {\mathcal{L}_{B_k}^{(k)}} \geq \det {B_k},}
	\end{equation}
	{and} 
	\begin{equation}\label{eq:numeri_romani_2}
	{\tag{II}}
	\mathcal{L}^{(k)}_{B_k}\s_k=B_k\s_k ,
	\end{equation} 
	then the new $\mathcal{L}^{(k)}$QN method is sound (see Algorithm~\ref{convergentlkqn_no_fq}) if the objective function is convex and has a minimizer (see Theorem~\ref{Sconvergence}, Theorem~\ref{directionpreservingprojection} and Corollary~\ref{remark:ontotallynonlinear}).
	
	The $\mathcal{L}^{(k)}$QN methods so obtained turn out to be a remarkable refinement of the methods introduced in \cite{CDTZ}. Observe that equation \eqref{eq:numeri_romani_2}, which allows to mimic the $BFGS$ self correction properties (see Section \ref{selfcorrectinanalysis}), is equivalent to the equality $(\mathcal{L}^{(k)}_{B_k})^{-1}\g_k=B_k^{-1}\g_k$, i.e., the {new introduced} method (Algorithm \ref{convergentlkqn_no_fq}) belongs to \textit{both}{ the $\mathcal{S}$ecant and $\mathcal{N}$on $\mathcal{S}$ecant class of Broyden Class-type methods (see {\cite{DFLZ,CDTZ}} and Section \ref{themetods} for the precise definitions)}, thus rising a question on the very meaning of secant equation in Quasi-Newton methods \cite{CDTZ}.}
{Moreover, developing a further adaptive criterion (see \eqref{h0condion}) for the choice of  $\tilde{B}_k=\mathcal{L}^{(k)}_{B_k}$, we produce a low complexity convergent $\mathcal{L}^{(k)}$QN  with quadratic termination property (see Algorithm \ref{convergentlkqn}). }

The {proposed} adaptive criteria can be satisfied by $\mathcal{L}^{(k)}= \sd U_k$ where $U_k$ is the  product of {three} Householder matrices. Algorithm~\ref{convergentlkqn_no_fq} and Algorithm~\ref{convergentlkqn} can be implemented by storing, respectively, $15$ or $17$ vectors of length $n$, whereas $L$-$BFGS$ -- a limited memory version of $BFGS$ suitable to solve large scale problems \cite{LN,numopt} -- requires $2M+2$ vectors of length $n$ (being $M$ the number of $\s_j,\y_j$ used to define $B_{k+1}$). Even if $L$-$BFGS$ is usually used with small values of $M$, it is well known that for some problems (see for example \cite{jiang2004preconditioned}) a greater value of $M$ could be required, and hence, for these problems, the memory required for the implementation of the algorithms here proposed could be considerably smaller. 
Note, moreover, that in contrast with $L$-$BFGS$ where some information is discarded at each step,  in Algorithm \ref{convergentlkqn_no_fq} and Algorithm \ref{convergentlkqn} the second order information generated in all the previous steps is stored in an approximate way.

Using performance profiles \cite{dolan2002benchmarking}  based on iterations, function evaluations and time, the results of numerical experiences on set of problems, taken from CUTEst \cite{gould2015cutest}, are provided. These experiences confirm that the proposed scheme (Algorithm 5) permits to guarantee a better level of approximation of second order information if compared with $L$-$BFGS$ (even if a big value of $M$ is chosen) resulting on an increased robustness. Additional numerical experiences on a different set of problems, see Experiment 2, highlight the competitiveness of our proposals if compared to previous $\mathcal{L}QN$ algorithms studied in literature. 
	
Moreover, following the ideas developed, for instance, in \cite{al1993analysis,al1998global}, a suitable scaling improves the efficiency of $\mathcal{L}^{(k)}QN$. In particular, Scaled (Sc) $\mathcal{L}^{(k)}$QN turns out to be competitive, in some cases, with respect to $L$-$BFGS$ (see Remark \ref{remark:scaling} and Section \ref{sec:numerical_results}). 
\section{Notation and preliminaries}
	We will freely use familiar properties of symmetric positive definite matrices and fundamental results concerning algebras of matrices simultaneously diagonalized by a given unitary transform. 
	
	We use the shorthand pd to denote a real symmetric positive definite matrix. Given a vector $\mathbf{z}\in \mathbb{R}^n$ we write $\mathbf{z}>0$ to denote entry-wise positivity. Let $d(\z)$ be the diagonal matrix whose diagonal entries are the components of $\z$; {let $d(A)$ and $\lambda(A)$ be the vectors of the diagonal entries and of the eigenvalues of a given matrix $A$, respectively.} Finally, the symbol $\| \cdot \|$ will denote both the euclidean norm for vectors and the corresponding induced norm for matrices. 
	
\subsection{Matrix Algebras} \label{matrixalgebras}
Let $M_n({\mathbb{C}})$ be the set of all $n \times n $ matrices with complex entries. Given a unitary matrix  $U\;\in\;M_n({\mathbb{C}})$ (i.e. U $n \times n $ and $U^H=U^{-1}$), {define the following algebra $\mathcal{L}$ of matrices:}
\begin{equation}\label{eq:sdu_definition}
\mathcal{L} := \sd\,U=\{Ud(\z)U^H\;:\;\z\;\in\;\mathbb{C}^n\}.
\end{equation}

\noindent {Given a matrix $B \in M_n({\mathbb{C}})$, by the Hilbert projection theorem, there exists a unique element $\mathcal{L}_B \in \mathcal{L}$ such that \begin{equation}
	||\mathcal{L}_B-B||_F\leq ||X-B||_F,\;\;\forall \;X\;\in\;\mathcal{L},
	\end{equation} where $\|\cdot\|_F$ denotes the Frobenius norm. It is easy to find the following explicit formula for $\mathcal{L}_B$ (see for example \cite{DFLZ}):
	\begin{equation} \label{eq:proj_expressionref}
	\mathcal{L}_B=Ud(\z_B)U^H, \hbox{ where } [\z_B]_i=[U^HBU]_{ii},\; i=1,\dots,n.
	\end{equation} 
	$\mathcal{L}_B$ will be called the \textit{best approximation in Frobenius norm} of $B$ in $\mathcal{L}$.}

For the sake of completeness we recall hereafter few important results on the projection $\mathcal{L}_B$ of a matrix $B$ onto a $\sd U$ subspace.  
\begin{lemma} \label{repspecprop}
	{Let $U$ be an unitary matrix, let $\mathcal{L}=\sd{U}$ and let $B\;\in\;M_n(\mathbb{C})$.} 
	\begin{enumerate}
		\item If $B=\x\y^T$, then $z_{\x\y^T}=d(U^H\x)U^T\y$ where $\x,\y \in \mathbb{C}^{n}$.
		\item If $B=B^H$, then $\mathcal{L}_B=\mathcal{L}_B^H$ and $\min\, \lambda (B)\leq \lambda (\mathcal{L}_B)\leq \max\, \lambda (B)$ where $\lambda(X)$ denotes the generic eigenvalue of $X$. Therefore $\mathcal{L}_B$ is Hermitian positive definite whenever $B$ is Hermitian positive definite. 
		\item If $B\in\mathbb{R}^{n \times n}$ then $\mathcal{L}_B\in\mathbb{R}^{n \times n}$ whenever {$\mathcal{L}$ is closed under conjugation (i.e., $A \in \mathcal{L} \Rightarrow \overline{A} \in \mathcal{L}$)}.
		{		\item $ \tr(\mathcal{L}_B)= \tr (B)$
			\item If $B$ is pd, then $\det(B) \leq \det(\mathcal{L}_B)$ where the equality holds iff $U$ diagonalizes B, i.e., iff $U^HBU$ is diagonal.}
	\end{enumerate}
\end{lemma}


\begin{proof}
	{For 1. see \cite{DFLZ}, for 2., 3. and 4. see {Propositions 5.2 in } \cite{DZ}. Concerning 5., let $A$ be a pd matrix. Then we have $\det {A}\leq \prod_{i=1}^na_{ii}$ (Hadamard inequality, see \cite{horn2ndmatrix}), and $\det(A)=\prod_{i=1}^na_{ii}$ if and only if $A$ is diagonal (see Theorem 7.8.1 \cite{horn2ndmatrix} ). In order to obtain 5. it is sufficient to apply these remarks to the pd matrix $U^HBU$. In fact, we have 
		\begin{equation*}
		\det(B)=\det(U^HBU)\leq \prod_{i=1}^n(U^HBU)_{ii}=\det(\mathcal{L}_B)
		\end{equation*} 
		and equality holds if and only if $U^HBU$ is diagonal.}
\end{proof}
\noindent The properties 4. and 5. of Lemma~\ref{repspecprop} will be crucial to state the conditions \eqref{traceconvergentsecantineq} and \eqref{determinantconvergentsecantineq}, for the convergence of the new method (see Theorem~\ref{Sconvergence}).

For a more exhaustive treatment of the contents of Lemma~\ref{repspecprop}, and its relevance for $\mathcal{L}^{(k)}$QN minimizations algorithms and optimal preconditioning of linear systems, one can see \cite{DFLZ}  and \cite{DZ}.
{Even if in the following sections we will use real unitary matrices $U$, in many situations the transform $U$ that diagonalizes matrices of $\mathcal{L}$, is defined on $\mathbb{C}$. This is the typical case of circulant matrices, where $U$ is the Fourier transform.  Then, to maintain a suitable degree of generality, the notation $U^H$ is necessary instead of $U^T$, and \textit{partial} results of the computational process, implicit in the iteration step 
	$B_k=\Phi(\mathcal{L}^{(k)}_{B_k}, \s_k, \y_k, \phi)$ (see Algorithm \ref{convergentlkqn_no_fq}), will be complex numbers. This does not compromise the fact that in each instruction the \textit{final} numerical results  are real. However, in this paper we will consider just real transforms $U$, so we will exchange the word `unitary' with the word `orthogonal' and the superscript `$H$' (Hermitian) with the superscript `$T$' (transpose) from the next section on.\\
	The algebras $\mathcal{L}$ considered in this article will be of low complexity, i.e., the matrix vector product $A\x$, for $A \in \mathcal{L}$, will be computable in a number of operations which grows slower than $O(n^2)$.}

\subsection{Broyden Class-type methods} \label{themetods}
Let us consider a function $f\, :\, \mathbb{R}^n \rightarrow \mathbb{R}$ where $n \geq 2$.

\noindent In this paper we will study the following class of minimization methods obtained by generalizing the Broyden Class methods considered in \cite{BNY}:

\IncMargin{1em}
\begin{algorithm}[H]
	\LinesNumbered
	
	\SetKwData{Left}{left}\SetKwData{This}{this}\SetKwData{Up}{up}
	
	\SetKwFunction{Union}{Union}\SetKwFunction{FindCompress}{FindCompress}
	
	\SetKwInOut{Input}{input}\SetKwInOut{Output}{output}
	
	\KwData{$\x_0\,\in \mathbb{R}^n, \,\g_0=\nabla f(\x_0),\, {\tilde{B}_0}$ pd, ${B_0}$ pd, $\mathbf{d}_0=-B_0^{-1}\g_0$, $k=0$; \\}
	
	\While{$\g_k\neq 0$}{

		$\x_{k+1}=\x_k+\lambda_k \d_k$ \tcc*[r]{$\lambda_k$ verifies conditions \eqref{AG1}, \eqref{AG2}}
		
		$\s_k=\x_{k+1}-\x_k $\;
		$\g_{k+1}=\nabla f(\x_{k+1})$\;
		
		$\y_k=\g_{k+1}-\g_k $\;
		
		$B_{k+1}= \Phi(\tilde{B}_k, \s_k,\y_k, \phi)$ \;
		$
				\left\{ 
				\begin{array}{c}
				\hbox{Define } \tilde{B}_{k+1} \hbox{ pd, set } \mathbf{d}_{k+1}=-\tilde{B}_{k+1}^{-1}\g_{k+1} \;\;\;(\mathcal{NS}) \; \\ 
				\hbox{Set } \mathbf{d}_{k+1}=-B_{k+1}^{-1}\g_{k+1}, \; \hbox{ define } \tilde{B}_{k+1} \hbox{ pd } \;\;\;(\mathcal{S})\; 
				\end{array}
				\right.
			$\;
		Set $k:=k+1$ \;	}\caption{Broyden Class-type}\label{gBc}
\end{algorithm}\DecMargin{1em}

\noindent	where $\tilde{B}_k$ is an approximation of $B_k$ and the updating formula is the Broyden's one applied to $\tilde{B}_k$, i.e.
\begin{equation}\label{Broydenupdate}
\Phi(\tilde{B}_k, \s_k,\y_k, \phi):=\tilde{B}_k- \frac{\tilde{B}_k\s_k\s_k^T\tilde{B}_k}{\s_k^T\tilde{B}_k\s_k}+\frac{\y_k\y_k^T}{\y_k^T\s_k}+\phi \, \s_k^T\tilde{B}_k\s_k\v_k\v_k^T.
\end{equation}
In \eqref{Broydenupdate} the vector $\v_k$ is defined by
$$\v_k=\frac{\y_k}{\y_k^T\s_k}-\frac{\tilde{B}_k\s_k}{\s_k^T\tilde{B}_k\s_k}$$ 

\noindent	and $\phi$ is a non negative parameter so that $\Phi(\tilde{B}_k, \s_k,\y_k, \phi)$ is pd whenever $\tilde{B}_k$ is pd and $\y_k^T\s_k>0$.

{ For $\phi \in [0,1]$ we call the Broyden Class-type family  ``restricted''.
	If $\tilde{B}_k=B_k$ for all $k$, then for $\phi=0$ and $\phi=1$ one obtains, respectively, the $BFGS$ and the DFP method {\cite{numopt}}.}

We assume that the step-length parameter $\lambda_k$ is chosen by an inexact line search satisfying the  {Wolfe} conditions
\begin{equation}\label{AG1}
f(\x_k+\lambda_k\mathbf{d}_k)\leq f(\x_k)+\alpha \lambda_k \g_k^T\mathbf{d}_k 
\end{equation} 
\begin{equation} \label{AG2}
g(\x_k+\lambda_k\mathbf{d}_k)^T\mathbf{d}_k\geq \beta \g_k^T\mathbf{d}_k 
\end{equation} 
where $0<\alpha < 1/2$ and $ \alpha < \beta < 1$. Condition \eqref{AG2} implies $\y_k^T\s_k>0$.\\

\noindent Let us observe that in the $\mathcal{S}$ case of Algorithm \ref{gBc},   the matrices generating the search directions $\mathbf{d}_{k+1}$ satisfy the Secant Equation $B_{k+1}\s_k=\y_k $. Instead, in the $\mathcal{NS}$ case such property is not necessarily fulfilled, i.e., in general, $\tilde{B}_{k+1}\s_k \neq \y_k.$

In the following three remarks we collect some useful properties we will use in Section \ref{sectionsconvergence}.\\

\begin{remark} \label{tracebelowbound}
	Observe that
	\begin{equation}\label{trace}
	\begin{split}
	\tr(B_{k+1})=\tr(\Phi(\tilde{B}_k, \s_k,\y_k, \phi))= \tr(\tilde{B}_k)+ \frac{\|\y_k\|^2}{\y_k^T\s_k}+\phi\frac{\|\y_k\|^2}{\y_k^T\s_k}\frac{\s_k^T\tilde{B}_k\s_k}{\y_k^T\s_k} \\ -(1-\phi)\frac{\|\tilde{B}_k\s_k\|^2}{\s_k^T\tilde{B}_k\s_k}-2\phi\frac{\y_k^T\tilde{B}_k\s_k}{\y_k^T\s_k}.
	\end{split}
	\end{equation}

	\noindent	Since $\phi \, \s_k^T\tilde{B}_k\s_k \geq 0$, the last term in \eqref{Broydenupdate} increases the eigenvalues {of the previous part of the update}, and hence
	\begin{equation}\label{equ:bc_bfgs_determinant}
	\det(B_{k+1}) \geq \det(\tilde{B}_k- \frac{\tilde{B}_k\s_k\s_k^T\tilde{B}_k}{\s_k^T\tilde{B}_k\s_k}+\frac{\y_k\y_k^T}{\y_k^T\s_k})=\det(\tilde{B}_k)\frac{\y_k^T\s_k}{\s_k^T\tilde{B}_k\s_k}
	\end{equation}
	(for the last equality see \cite{numopt}).
\end{remark}

\begin{remark}\label{AG2bis}
	From \eqref{AG2} it follows that, {using definitions in Algorithm \ref{gBc}},
	\begin{equation} \label{yslowerbound}
	\y_k^T\s_k=\g_{k+1}^T\s_k-\g_{k}^T\s_k\geq -(1-\beta)\g_k^T\s_k
	\end{equation}
	from which we obtain
	\begin{equation}
	\frac{\s_k^T\tilde{B}_k\s_k}{\y_k^T\s_k} \leq \frac{\s_k^T\tilde{B}_k\s_k}{(1-\beta)(-\g_k^T\s_k)}=\frac{\lambda_k}{1-\beta}
	\end{equation}
	($\s_k^T\tilde{B}_k\s_k=\s_k^T(-\lambda_k\g_k)$ in the $\mathcal{NS} $ case) and
	\begin{equation}
	\frac{\s_k^T{B}_k\s_k}{\y_k^T\s_k} \leq \frac{\s_k^T{B}_k\s_k}{(1-\beta)(-\g_k^T\s_k)}=\frac{\lambda_k}{1-\beta}
	\end{equation}
	($\s_k^T{B}_k\s_k=\s_k^T(-\lambda_k\g_k)$ in the $\mathcal{S} $ case).
\end{remark}

\begin{remark} \label{zeroconvergencesg}
	Let us define $f_*$ to be the infimum of $f$. 
	Using \eqref{AG1} we have (in both $\mathcal{NS}$ and $\mathcal{S}$ methods)
	\begin{equation}
	\begin{split}\sum_{k=0}^{N} \s_k^T(-\g_k)=\sum_{k=0}^{N} -\lambda_k\mathbf{d}_k^T\g_k \\
	\leq \frac{1}{\alpha}\sum_{k=0}^{N} [f(\x_k)-f(\x_{k+1})] \\
	\leq \frac{1}{\alpha}[f(\x_0)-f_*] <\infty.
	\end{split}
	\end{equation}
	Then the sum converges for $n \rightarrow +\infty$,	from which we obtain 
	$$\lim_{k \to +\infty}\s_k^T(-\g_k)=0.$$
\end{remark}

\subsection{{Assumptions for the function $f$}} \label{assumptionsection}
In Section \ref{sectionsconvergence}, in order to obtain a convergence result for the Broyden Class-type, we will {do the following}:
\begin{assumption} \label{assumprionf1}
	The level set $$D=\{\x\,\in\,\mathbb{R}^n\, : f(\x)\leq f(\x_0)\}$$ is convex, the function $f(\x)$  is  twice continuously differentiable, convex and bounded below in $D$ and the
	Hessian matrix is bounded in $D$, i.e.
	\begin{equation} \label{hessianbound}
	\|G(\x)\| \leq M,  {\hbox{ being } M \hbox{ a positive contant.}}
	\end{equation}
\end{assumption}

\begin{remark} \label{genassumption} 
	Observe that the condition \eqref{hessianbound} could be replaced by uniform convexity of $f(\x)$ and Lipschitz condition on $G(\x)$. Moreover, if Assumption \ref{assumprionf1} is fulfilled, then the following boundedness condition on the the Powell's ratio  ${\|\y_k\|^2}/{\s_k^T\y_k}$ \cite{P}  holds:
	
	\begin{equation} \label{ratiobound}
	\frac{\|\y_k\|^2}{\s_k^T\y_k} \leq M.
	\end{equation}  	
	{where $\s_k$, $\y_k$ are the difference vectors produced by Algorithm \ref{gBc}.}
	In fact, if we define (see \cite{BNY}, \cite{numopt}) the pd matrix
	\begin{equation} \label{eq:mean_hessian}
	\overline{G}=\int_0^1G(\x_k+\tau \s_k)d\tau,
	\end{equation}
	then we have from standard analysis results,
	\begin{equation} \label{hessiansecant}
	\y_k=\overline{G}\s_k
	\end{equation}
	and hence if $\z_k=\overline{G}^{\frac{1}{2}}\s_k$,
	\begin{equation*} 
	\frac{\|\y_k\|^2}{\s_k^T\y_k}=\frac{\s_k^T\overline{G}^2\s_k}{\s_k^T\overline{G}\s_k}= \frac{\z_k^T\overline{G}\z_k}{\z_k^T\z_k}\leq \sup_{\tau\;\in\;[0,1]} \|G(\x_k+\tau \s_k)\| \leq M.
	\end{equation*}	
\end{remark}
\noindent	We recall that condition \eqref{ratiobound}  is  typically used to prove  the global convergence of $BFGS$ method  \cite{P} and of  $\mathcal{L}QN$ methods \cite{DFLZ}.
{Observe that, if one could impose the \textit{discrete} convexity condition \eqref{ratiobound} by a suitable line-search, the  convergence results in the following sections would hold under the weaker assumptions $f \in C^1$ and bounded below.}

\section{Conditions for the convergence of the $\mathcal{S}$ecant and $\mathcal{N}$on $\mathcal{S}$ecant Broyden Class-type} \label{sectionsconvergence}
The matrices which generate the descent directions in the $\mathcal{S}$ case  exhibit explicitly  second order information (or, in other words, they satisfy the secant equation). Moreover, in contrast with the limited memory versions of Quasi-Newton methods, they store, in an approximate way, the second order information generated in all the previous steps of the algorithm.
In this section we will prove that both $\mathcal{S}$ and $\mathcal{NS}$ versions of Algorithm 1 are convergent if $\tilde{B}_k$ is suitably chosen.

\noindent{Now, using techniques and ideas developed in \cite{BNY,byrd1989tool}, we state the following 
	result which generalizes to the Broyden class of updating formulas \cite{BNY} what 
	has been proved in \cite{CDTZ} for $BFGS$-type $\mathcal{S}$ methods.}

\begin{theorem} \label{Sconvergence}
	If the $\mathcal{S}$ version  of Algorithm~\ref{gBc} with $\phi\;\in\;[0,1)$  is applied to a function that satisfies Assumption {\ref{assumprionf1}} and $\tilde{B}_k$ is chosen such that

	\begin{equation} \label{traceconvergentsecantineq}
	\tr \tilde{B}_k \leq \tr B_k
	\end{equation}
	\begin{equation} \label{determinantconvergentsecantineq}
	\det \tilde{B}_k \geq \det B_k
	\end{equation}
	\begin{equation} \label{directionpreservingineq}
	\frac{||B_k\s_k||^2}{(\s^T_kB_k\s_k)^2} \leq \frac{||\tilde{B}_k\s_k||^2}{(\s_k^T\tilde{B}_k\s_k)^2} \,.
	\end{equation}   
	
	\noindent 	for all $k$, then
	\begin{equation} \label{liminf}
	\liminf_{k \to \infty} \|\g_k\|=0
	\end{equation}	
	for any starting point $\x_0$ and any {pd} matrix $B_0$.
\end{theorem}

\noindent The main idea to prove Theorem~\ref{Sconvergence} is to compare the third and fifth term of \eqref{trace}. Let us define $\psi_k$  {as}

\begin{equation}\label{psiequation}
{
	\psi_k:= 
	\big [{\frac{\|\y_k\|^2}{\y_k^T\s_k}\frac{\s_k^T\tilde{B}_k\s_k}{\y_k^T\s_k} -2\frac{\y_k^T\tilde{B}_k\s_k}{\y_k^T\s_k}}\big ]{\frac{\s_k^T\tilde{B}_k\s_k}{\|\tilde{B}_k\s_k\|^2}}}
\end{equation}
so that  \eqref{trace} becomes
\begin{equation} \label{psitraceequation}
\tr(B_{k+1})= \tr(\tilde{B_k})+ \frac{\|\y_k\|^2}{\y_k^T\s_k} -(1-\phi-\psi_k \phi)\frac{\|\tilde{B}_k\s_k\|^2}{\s_k^T\tilde{B}_k\s_k}.
\end{equation}


\noindent In the following remarks we state upper bounds for the addends appearing in~\eqref{psiequation}. 

\begin{remark}\label{firsttermpsidis}
	\begin{equation} \label{eq:firsttermpsidis}
	\begin{split}
	\frac{\|\y_k\|^2}{\y_k^T\s_k}\frac{\s_k^T\tilde{B}_k\s_k}{\y_k^T\s_k}{\frac{\s_k^T\tilde{B}_k\s_k}{\|\tilde{B}_k\s_k\|^2}} 
	\leq M \frac{(\s_k^T\tilde{B}_k\s_k)^2}{\y_k^T\s_k\|\tilde{B}_k\s_k\|^2} \\
	\leq M \frac{(\s_k^T{B}_k\s_k)^2}{\y_k^T\s_k\|{B}_k\s_k\|^2}
	=\frac{M(\s_k^T(-\g_k))^2}{\y_k^T\s_k\|-\g_k\|^2} \\
	\leq \frac{M(\s_k^T(-\g_k))}{(1-\beta)\|-\g_k\|^2},
	\end{split}
	\end{equation}
	where first inequality follows using \eqref{ratiobound}, the second using \eqref{directionpreservingineq} and last inequality follows using \eqref{yslowerbound}.
\end{remark}
\begin{remark}\label{secondtermpsidis}
	\begin{equation} \label{eq:secondtermpsidis}
	\begin{split}
	\frac{|\y_k^T\tilde{B}_k\s_k|}{\y_k^T\s_k}{\frac{\s_k^T\tilde{B}_k\s_k}{\|\tilde{B}_k\s_k\|^2}} \leq \frac{\|\y_k\|\s_k^T\tilde{B}_k\s_k}{\y_k^T\s_k\|\tilde{B}_k\s_k\|} \\
	\leq \frac{\sqrt{M}\s_k^T\tilde{B}_k\s_k}{\sqrt{\y_k^T\s_k}\|\tilde{B}_k\s_k\|}\\
	\leq \frac{\sqrt{M}\s_k^T{B}_k\s_k}{\sqrt{\y_k^T\s_k}\|{B}_k\s_k\|}
	= \frac{\sqrt{M}(\s_k^T(-\g_k))}{\sqrt{\y_k^T\s_k}\|-\g_k\|} \\
	\leq \frac{\sqrt{M(\s_k^T(-\g_k))}}{\sqrt{1-\beta}\|-\g_k\|},
	\end{split}
	\end{equation}
	where the first inequality follows from Cauchy-Schwarz inequality, the second from \eqref{ratiobound}, the third from \eqref{directionpreservingineq}, the fourth from \eqref{yslowerbound}.
\end{remark}
We can now prove Theorem~\ref{Sconvergence}.
\begin{proof}
	Arguing by contradiction, let us assume  $\|\g_k\|$ bounded away from zero, i.e., there exists $\gamma >0$ such that
	\begin{equation} \label{absurdhyp}
	\|\g_k\|\geq \gamma >0.
	\end{equation}
	From Remark \ref{zeroconvergencesg}  we obtain  \begin{equation}\label{absurdconsequence}
	\lim_{k \to \infty}\frac{\s_k^T(-\g_k)}{\|-\g_k\|^2} =0.
	\end{equation}
	Now we show that \eqref{absurdconsequence} leads to a contradiction, thus \eqref{absurdhyp} cannot hold.
	From \eqref{psiequation}, using Remark \ref{firsttermpsidis}, Remark \ref{secondtermpsidis} and \eqref{absurdconsequence} we obtain
	\begin{equation} \label{psitendtozero}
	\lim_{k \to \infty} \psi_k=0.
	\end{equation}
	Using \eqref{psitendtozero}, since $\phi \in [0, 1)$, we have that there exist an index $s$ and constants $l_1>0,\;l_2>0$ such that 
	\begin{equation} \label{ineq:phi_psi}
	l_2 \geq (1-\phi-\psi_k \phi) \geq l_1 >0\; \hbox{ for all } \;k \geq s.
	\end{equation}
	Then we can write (for $j \geq s$), {using \eqref{psitraceequation}},
	\begin{equation}\label{eqq2 : nw}
	{\rm tr} B_{j+1}\leq {\rm tr} B_s + \sum_{k=s}^{j}\frac{1}{\y_k^T\s_k}\|\y_k\|^2-\sum_{k=s}^{j}\frac{1}{\s_k^T\tilde B_k \s_k}\|\tilde B_k\s_k\|^2(1-\phi-\psi_k \phi),
	\end{equation}
	and hence
	
	\begin{equation} \label{eq:referee_added}
	{\rm tr} B_{j+1}\leq {\rm tr} B_s + \sum_{k=s}^{j}\frac{1}{\y_k^T\s_k}\|\y_k\|^2\leq {\rm tr} B_s + M(j+1-s) \leq c_1 (j+2-s)
	\end{equation}
	
	where $c_1=\max\{\tr B_s, M\}$ (the trace grows at most linearly for all $j \geq s$).
	
	Let us remember that, given $n$ real positive numbers $a_i$, it holds:
	\begin{equation}\label{ag : nw}
	\prod_{i=1}^{n} a_i\leq \bigg (\frac{\sum_{i=1}^na_i}{n} \bigg )^n
	\end{equation}
	from which we obtain:
	\begin{equation}\label{eqq4 : nw}
	\det B_{j+1}=\prod_{i=1}^n \lambda_i(B_{j+1}) \leq \bigg (\frac{\sum_{i=1}^n\lambda_i(B_{j+1})}{n} \bigg )^n \leq \bigg (\frac{c_1(j+2-s)}{n} \bigg)^n \ .
	\end{equation}
	Let us note, moreover, that from (\ref{eqq2 : nw}) and \eqref{eq:referee_added}, since $B_{j+1}$ is  {pd}, we have:
	\begin{equation}\begin{split}
	\sum_{k=s}^{j}\frac{1}{\s_k^T\tilde B_k \s_k}\|\tilde B_k\s_k\|^2 (1-\phi-\psi_k \phi)\leq  {\rm tr} B_s - {\rm tr} B_{j+1} +  \sum_{k=s}^{j}\frac{1}{\y_k^T\s_k}\|\y_k\|^2 \\
	\leq  {\rm tr} B_s + \sum_{k=s}^{j}\frac{1}{\y_k^T\s_k}\|\y_k\|^2 \leq c_1(j+2-s)
	\end{split}
	\end{equation}
	and applying once more (\ref{ag : nw}) we have:
	\begin{equation}\label{eqq5 : nw}
	\prod_{k=s}^{j}\frac{1}{\s_k^T\tilde B_k \s_k}\|\tilde B_k\s_k\|^2 (1-\phi-\psi_k \phi)\leq (2c_1)^{j+1-s}.
	\end{equation}
	
	\noindent From \eqref{equ:bc_bfgs_determinant} and \eqref{determinantconvergentsecantineq} we have:
	$$
	\det B_{j+1} \geq \frac{\s_j^T\y_j}{\s_j^T\tilde B_j \s_j}\det \tilde B_j \ge \frac{\s_j^T\y_j}{\s_j^T\tilde B_j \s_j} \det B_j, 
	$$
	from which we obtain:
	\begin{equation}\label{eqq6 : nw}
	\prod_{k=s}^{j}\frac{\s_k^T\y_k}{\s_k^T\tilde B_k \s_k}\leq \frac{\det B_{j+1}}{\det B_s} \ .
	\end{equation}
	From \eqref{yslowerbound} we have
	$$
	(1-\beta)^{j+1-s} \leq \prod_{k=s}^{j}\frac{\s_k^T\y_k}{-\g_k^T\s_k},
	$$
	and hence, by the equality $B_k\s_k=-\lambda_k\g_k$ and by \eqref{directionpreservingineq},  \eqref{eqq4 : nw}, \eqref{eqq5 : nw}, \eqref{eqq6 : nw},
	\begin{equation} \label{ineqintheproof} 
	\begin{split}
	(1-\beta)^{j+1-s}\prod_{k=s}^{j}\frac{\|\g_k\|^2}{\s_k^T(-\g_k)}(1-\phi-\psi_k \phi) \\\leq \prod_{k=s}^{j}(1-\phi-\psi_k \phi)\frac{\|-\lambda_k\g_k\|^2}{\s_k^T(-\lambda_k\g_k)}\frac{\s_k^T\y_k}{\s_k^T(-\lambda_k\g_k)}\\
	=\prod_{k=s}^{j}(1-\phi-\psi_k \phi)\frac{\|B_k\s_k\|^2}{\s_k^T B_k\s_k}\frac{\s_k^T\y_k}{\s_k^T\ B_k\s_k}\\
	\leq \prod_{k=s}^{j}(1-\phi-\psi_k \phi)\frac{\|\tilde B_k\s_k\|^2}{\s_k^T\tilde B_k\s_k}\frac{\s_k^T\y_k}{\s_k^T\tilde B_k\s_k}\\
	\leq (2c_1)^{j+1-s}\bigg(\frac{c_1(j+2-s)}{n}\bigg)^n\frac{1}{\det B_s},
	\end{split}
	\end{equation}
	i.e.,
	\begin{equation}\label{eqq7 : nw} 
	\prod_{k=s}^{j}(1-\phi-\psi_k \phi)\frac{\|\g_k\|^2}{\s_k^T(-\g_k)}\leq c_2^{j+1-s} \hbox{ for all } j\geq s,
	\end{equation}
	{for a suitable constant $c_2$ dependent on $s$ and $M$ (defined in Assumption \ref{assumprionf1}). For the details see Appendix 2.}
	
	\noindent On the other hand, by \eqref{absurdconsequence} and by the bound $1-\phi-\psi_k\phi \geq l_1 > 0$  in \eqref{ineq:phi_psi},
	we have that the ratios $(1-\phi-\psi_k\phi)\|\g_k\|^2/\s_k^T(-\g_k)$ go to $+\infty$, as $k\to +\infty$; thus a natural number $j^*\geq s$ must exist such that
	
	\begin{equation}\label{eqq7 : nw_2}
	\tag{$\ref{eqq7 : nw}_1$}
	\prod_{k=s}^j (1-\phi-\psi_k\phi)\mathbf{f}{ \|\g_k\|^2 }{ \s_k^T(-\g_k) }
	> c_2^{j+1-s},
	\ \ \forall\ j\geq j^* ,
	\end{equation}
	({see again Appendix 2 for more details}) but this contradicts \eqref{eqq7 : nw} {choosing $j \geq \max\{s,j^*\}$}.  We have hence proved that \eqref{liminf} holds.
\end{proof}
\noindent The condition \eqref{directionpreservingineq} is satisfied, in particular, when $\tilde{B}_k$ is such that
\begin{equation} \label{directionpreserving}
\tilde{B}_k\s_k=B_{k}\s_k.
\end{equation} 
{In the following the above equality has  a crucial role. As it is clear from Algorithm \ref{gBc}, the equality \eqref{directionpreserving} 
	regards the basic relationship between the search directions produced by $\mathcal{S}$ and $\mathcal{NS}$ algorithms. In fact, if equality \eqref{directionpreserving} holds, such search directions  are perfectly equivalent even if $B_k \neq \tilde{B}_k$. To prove the convergence property of the $\mathcal{S}$ scheme we have exploited the condition \eqref{directionpreservingineq}, which is fulfilled if \eqref{directionpreserving} is fulfilled.}\\
\noindent In the next Sections \ref{selfcorrectinanalysis} and \ref{constructionSconvegence} we will investigate some further
consequences of condition \eqref{directionpreserving} and we will prove that it can be imposed
by choosing $\tilde{B}_k$ as the projection of $B_k$ onto algebras of matrices diagonalized by {a fixed, small number of orthogonal} Householder transforms.

The following result generalizes what proven in \cite{DFLZ} for BFGS-type $\mathcal{NS}$ methods. 
\begin{theorem}
	If the $\mathcal{NS}$ version  of Algorithm~\ref{gBc} with $\phi\;\in\;[0,1)$  is applied to a function that satisfies Assumption {\ref{assumprionf1}} and $\tilde{B}_k$ is chosen such that 
	\eqref{traceconvergentsecantineq} and \eqref{determinantconvergentsecantineq} hold for all $k$,	then
	\begin{equation}
	\liminf_{k \to \infty} \|\g_k\|=0
	\end{equation}	
	for any starting point $\x_0$ and any {pd} matrix $B_0$.
\end{theorem}
\begin{proof} 
	Proceed as in the proof of Theorem~\ref{Sconvergence} noting that the hypothesis \eqref{directionpreservingineq} on $\tilde{B}_k$ is no longer necessary to obtain  Remark \ref{firsttermpsidis} (see \eqref{eq:firsttermpsidis}), Remark \ref{secondtermpsidis} (see \eqref{eq:secondtermpsidis}) and \eqref{ineqintheproof}, since in $\mathcal{NS}$ methods  $\tilde{B}_k\s_k$
	turns out to be equal to $-\lambda_k \g_k$. 
\end{proof}

\noindent	In Figure \ref{pictorialproperties} we illustrate in a pictorial way the restricted
Broyden Class-type $\mathcal{S}$ecant and $\mathcal{N}$on $\mathcal{S}$ecant methods satisfying
the conditions $\tr \tilde{B}_k \leq \tr B_k$, $\det \tilde{B}_k \geq \det B_k$ and {$f \in C^2$}, which appear basic in proving convergence
results for both classes of methods. At the moment only a subset
of the pictured $\mathcal{S}$ecant methods are certainly convergent, those
satisfying the surplus condition \eqref{directionpreservingineq}. 
\noindent In the following we will focus on Broyden Class-type methods
such that $\tilde{B}_k \s_k = B_k \s_k$, which form a subset of the intersection
between convergent $\mathcal{S}$ and $\mathcal{NS}$, with the aim to define new efficient $BFGS$-type algorithms. 

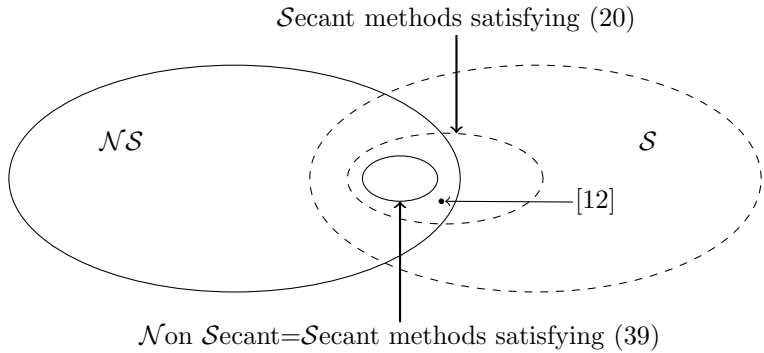
\begin{figure}[h] 
	\centering	
	\def\firstcircle{(0,0) ellipse (3 and 1.5)}
	\def\secondcircle{(4,0) ellipse (3 and 1.5)}
	\def\thirdcircle{(2.8,0) ellipse (1.3 and 0.6)}
	\def\fourthcircle{(2.2,0) ellipse (0.5 and 0.3)}
	\def\fifthcircle{(2.75,-0.3) ellipse (0.03 and 0.03)}
	\begin{tikzpicture}
	\draw \firstcircle  ;
	\draw[dashed] \secondcircle ;
	\draw[dashed] \thirdcircle ;
	\draw(-1.5,0.5) node {$\mathcal{NS}$};
	\draw(5.5,0.5) node {$\mathcal{S}$};
	\draw[<-,thick](2.95,0.6)--(2.95,1.9) ;
	\draw(2.95,2.1) node {$\mathcal{S}$ecant methods satisfying \eqref{directionpreservingineq} };
	\draw \fourthcircle;
	\draw \fifthcircle;
	\fill[black] \fifthcircle;
	\draw[<-,thick](2.2,-0.3)--(2.2,-1.9) ;
	\draw[<-](2.8,-0.3)--(4.5,-0.31) ;
	\draw(4.8,-0.3) node {\cite{CDTZ}} ;
	\draw(2.2,-2.1) node {$\mathcal{N}$on  $\mathcal{S}$ecant=$\mathcal{S}$ecant methods satisfying \eqref{directionpreserving}};
	
	\end{tikzpicture}	\caption{Restricted Broyden Class-type methods satisfying the conditions on trace, determinant.} \label{pictorialproperties}
\end{figure}

\section{Self correcting properties implied by convergence conditions} \label{selfcorrectinanalysis}
In this section, assuming $\phi=0$ in Algorithm \ref{gBc}, we will study how \eqref{directionpreserving}  reverberates on self correcting properties of the algorithm.

There are experimental evidences (in the case the matrix $\tilde{B}_k$ is chosen in some fixed matrix algebra $\mathcal{L}$), that the $\mathcal{S}$ version of Algorithm \ref{gBc} perform{s} better if compared with the $\mathcal{NS}$ one (see \cite{BDCFZ} and \cite{CCD}).
In this section we will try to motivate theoretically this experimental observation by comparing $\tr{B_{k+1}}$ and $\det{B_{k+1}}$ produced by classic $BFGS$ and Algorithm \ref{gBc} when $\phi=0$. 
Observe moreover, that in  \cite{CDTZ} some preliminary experimental experiences have shown that even if condition \eqref{directionpreserving} is imposed in an approximate way (i.e $\tilde{B}_k\s_k \approx B_k\s_k$) performances of Algorithm \ref{gBc} are competitive with those of $\mathcal{H}QN$, which, in turn, has been proved to be competitive with $L$-$BFGS$ on some neural networks problem (see \cite{DFLZ,BDCFZ}).

Finally let us stress the fact that, even if ``the Quasi-Newton updating is inherently an overwriting process rather than an averaging process" (see \cite{BHNS}), the following analysis will show how algorithms proposed in this work exhibit an interaction between averaging and overwriting phases more similar to $BFGS$ than to $L$-$BFGS$ (remember that the curvature information constructed by $BFGS$ are good enough to endow the algorithm with a superlinear rate of convergence, see \cite{numopt}).

Performing one step of the ``classic" $BFGS$, one has (see {\eqref{trace} and \eqref{equ:bc_bfgs_determinant}})

\begin{equation} \label{eq:ft_selfcorrecting}
\begin{split}
& B_{k+1}=\Phi(B_k, \s_k, \y_k,0) \\
&\tr B_{k+1}=\tr B_k {-\frac{\|B_k\s_k\|^2}{\s_k^TB_k\s_k}}+\frac{\|\y_k\|^2}{\y_k^T\s_k}
\end{split}
\end{equation}
\begin{equation} \label{eq:det_selfcorrecting}
\begin{split}
\det(B_{k+1})=\det ({B}_k) \frac{\y_k^T\s_k}{\s_k^TB_k\s_k}=\det ({{B}_k}) \frac{\s_k^T(\overline{G}\s_k)}{\s_k^TB_k\s_k},
\end{split}
\end{equation}
from which it is {possible to observe} that $BFGS$ (and all updates in the restricted Broyden class) ``have a strong self correcting property with respect to the determinant" (see \cite{BNY}). In particular curvatures of the model are inflated or deflated (and hence corrected) accordingly to the ratio $\frac{\s_k^T(\overline{G}\s_k)}{\s_k^TB_k\s_k},$
allowing the algorithm to compare the computed model with the true Hessian. In fact, {the previous} ratio is used to correct the spectrum of the operator defining the descent direction at next step.

On the contrary, by performing one step of Algorithm \ref{gBc}, we obtain {equations \eqref{eq:ft_selfcorrecting} and \eqref{eq:det_selfcorrecting}  where ${B}_k$ is replaced by $\tilde{B}_k$.}
It is {then} clear that if $\tilde{B}_k\s_k$ is not suitably chosen, then the ratio $\frac{\s_k^T(\overline{G}\s_k)}{\s_k^T\tilde{B}_k\s_k}$
could not exhibit a reasonable  behavior, making the algorithm not able to self-correct bad estimated curvatures and hence loosing efficiency.
{Hypothesis \eqref{directionpreserving} is hence further justified
	from the ``self-correcting properties point of view''.} 
Observe that if we choose $\tilde{B}_k=\mathcal{L}^{(k)}_{B_k}$, the error we introduce contributes to inappropriately inflate the curvatures of the model because by Lemma~\ref{repspecprop}, even if $\tr \tilde{B}_k=\tr B_k$, we have  $\det{\tilde{B}_k} \geq \det {B_k}$ (see \cite{liuwiel} and references therein for more information regarding the inappropriate inflations problems affecting $BFGS$).
Recall that by the same Lemma~\ref{repspecprop}, $\det{\tilde{B}_k} = \det {B_k}$ iff $U_k$ diagonalizes $B_k$. Thus, in order to reduce the inappropriate inflation of the curvatures of the model, $U_k$ should be chosen, in principle, besides of low complexity, as close as possible to a matrix which diagonalizes $B_k$. 

The problem concerning the possibility to exploit $\tilde{B}_k$ in order to improve such self correcting properties as much as possible remains open. Anyway, the relative weakness of the hypothesis of Theorem~\ref{Sconvergence} leaves room, in principle, for possible different choices of $\tilde{B}_k$, besides the specific choice considered in this work, which could improve self correcting property.

A quite natural choice of $\tilde{B}_k$, alternative to $\tilde{B}_k=\mathcal{L}^{(k)}_{B_k}$, can be $\tilde{B}_k=\sigma_k\mathcal{L}^{(k)}_{B_k}$ for a suitable $\sigma_k$, as considered in the following Remark \ref{remark:scaling} (see also \cite{al1993analysis,al1998global}).

\begin{remark} \label{remark:scaling}
	In section \ref{sec:numerical_results}, in order to mitigate the inappropriate inflation of the curvatures introduced by the projection operation,
	following a well known line of research \cite{al1993analysis,al1998global,nocedal1993analysis,oren1974self,andrei2018double}, we numerically investigate the introduction of a self-scaling factor $\sigma_k$, i.e., we will use $\tilde{B}_k=\sigma_k \mathcal{L}^{(k)}_{B_k}$.  More in detail, after the construction of the matrix algebra $\mathcal{L}^{(k)}$ such that $\mathcal{L}^{(k)}_{B_k} \s_k= {B}_k\s_k$ (see Section \ref{constructionSconvegence}, Line \ref{algline:B_kcomputation_no_qt} of Algorithm~\ref{convergentlkqn_no_fq} and Line \ref{algline:compute_proj2} of Algorithm~\ref{convergentlkqn}), 
	we  scale $\mathcal{L}^{(k)}_{B_k}$; in particular, we use  the updating formula
	
	\begin{equation} \label{eq:scaled_update}
	B_{k+1}=\Phi(\sigma_k \mathcal{L}^{(k)}_{B_{k}}, \s_k, \y_k,0)
	\end{equation}
	where
	$$\sigma_k:= \max\{ \min \{{\frac{\y_k^T\s_k}{\s_k^T\mathcal{L}^{(k)}_{B_k}\s_k}},1 \}, (\det(B_k)/\det(\mathcal{L}^{(k)}_{B_k}))^{1/n}\}.$$ Such choice of $\sigma_k$ guarantees that all the hypothesis of Theorem  \ref{Sconvergence} are satisfied. Moreover, as $\sigma_k \leq 1$ for all $k$, we have
	\begin{equation*}
	\det(B_{k+1})=\det(\sigma_k \mathcal{L}^{(k)}_{B_{k}}) \frac{\y_k^T\s_k}{\s_k^T\sigma_k\mathcal{L}^{(k)}_{B_{k}}\s_k} \leq \det( \mathcal{L}^{(k)}_{B_{k}}) \frac{\y_k^T\s_k}{\s_k^T\mathcal{L}^{(k)}_{B_{k}}\s_k},
	\end{equation*}
	which implies 
	\begin{equation*}
	\det(B_{k+1})=\det(\Phi(\sigma_k \mathcal{L}^{(k)}_{B_{k}}, \s_k, \y_k,0)) \leq  \det(\Phi(\mathcal{L}^{(k)}_{B_{k}}, \s_k, \y_k,0)),
	\end{equation*}
	i.e., the determinants of the matrices $B_{k+1}$ generated with the $\sigma_k$-scaled updating formula \eqref{eq:scaled_update} are smaller than the determinants of the matrices $B_{k+1}$ generated through the not scaled formula \eqref{eq:low_rank_unpdate_introduction} with $\phi=0$.
	
	In the experiments considered in Section \ref{sec:numerical_results}, the choice $\tilde{B}_k=\sigma_k \mathcal{L}^{(k)}_{B_{k}}$	turns out to improve in certain cases, the \textit{not-scaled} $\mathcal{L}^{(k)}QN$ methods and indicates a possible optimization strategy, based on $\mathcal{L}^{(k)}QN$, competitive with $L$-$BFGS$.

	Finally, let us observe that in \cite{CDTZ} the authors investigated $BFGS$-type methods where
	$\sigma_k \mathcal{L}^{(k)}_{{B}_k} \s_k = B_k \s_k$ for some $\sigma_k > 0$. Nevertheless, in \cite{CDTZ} $\sigma_k$ was a parameter used in the construction of the matrix algebra~$\mathcal{L}^{(k)}$. 	
\end{remark}

{\section{How to ensure $\mathcal{S}$ecant convergence conditions by low complexity matrices}\label{constructionSconvegence}}
In this section we will show that it is always possible to satisfy hypothesis of Theorem~\ref{Sconvergence} by a low complexity matrix $\tilde{B}_k$.
In particular, a matrix $\tilde{B}_k$ satisfying \eqref{traceconvergentsecantineq}, \eqref{determinantconvergentsecantineq} and \eqref{directionpreserving} will be explicitly constructed.

\noindent	As noticed in Lemma~\ref{repspecprop}, spectral conditions \eqref{traceconvergentsecantineq}, \eqref{determinantconvergentsecantineq} are always fulfilled when we choose 
$$\tilde{B}_k=\mathcal{L}_{B_k} \hbox{ for some } \mathcal{L}=\sd U.$$
{Nevertheless, the condition} 
\begin{equation}\label{eq:convergenceconditionprojection}
\mathcal{L}_{B_k} \s_k= {B}_k\s_k.
\end{equation}
{is not satisfied for a generic matrix algebra $\mathcal{L}$ and we have to face the following Problem \ref{TNLP} (see \cite{CDTZ} for an analogous problem involving a parameter $\sigma$): }
\begin{prob} \label{TNLP}
	Given a pd matrix $A \in \mathbb{R}^{n\times n}$  and a vector $\s \in \mathbb{R}^n$, find a low complexity orthogonal matrix ${U}$ such that 
	\begin{equation}
	{\mathcal{L}}_{A}s=As 
	\end{equation}
	where ${\mathcal{L}}=\sd {U}$.
\end{prob}
\noindent Observe that Problem \ref{TNLP} has been solved in \cite{CDTpagerank} in the particular case when $\s$ is an eigenvector of $A$ {with the aim} to speed-up the Pagerank computation by the preconditioned Euler-Richardson method. {The following Lemma~\ref{lemma:eigenvector_tnlp_solution}  completely characterizes solution of Problem \ref{TNLP} in this case.}
\begin{lemma} \label{lemma:eigenvector_tnlp_solution}
	{Le $A$ be a $n \times n$ symmetric matrix, if $\s$ is such that $A\s=\gamma \s $, then for any orthogonal matrix $L$ such that $\s/\|\s\|$ is among its columns, we have 
		\begin{equation*}
		{\mathcal{L}}_{A}s=As
		\end{equation*}
		where $\mathcal{L}=\sd L$. In particular $L$ can be chosen as an orthogonal Householder matrix.}
\end{lemma}  
\begin{proof}
	{Consider an orthogonal $L$ such that $L \e_{k}=\s/\|\s\|$ for some fixed $k \in \{1,\dots,n\}$. From \eqref{eq:proj_expressionref} we have $\mathcal{L}_A=Ld(\z_A)L^T$ being $\z_A$ the vector $$\z_A=[\dots,(L^TAL)_{ii}, \dots]^T,$$ 
		and hence
		\begin{equation}
		\mathcal{L}_A\s={(\z_A)_k}\s=\frac{\s^T A \s}{\|\s\|^2}\s={\gamma}\s=A\s.
		\end{equation} 
		For the second part see Lemma~\ref{lem:householder_fixed_columns} in the Appendix.}
\end{proof}
%

The following Theorem~\ref{directionpreservingprojection} solves Problem \ref{TNLP} in the general case and, at the same time,  sheds light on algorithmic details necessary for the construction of the solution. In \cite{CDZarnoldipreserver} it is solved a more general problem  where the projection $\mathcal{L}_A$ retains the action of $A$ on a set of vectors instead on a single one. \\

\noindent Let us begin recalling the well-known Arnoldi algorithm {\cite{saig}} for finding an orthogonal basis of the Krylov subspace $$\mathcal{K}_m(A,\v):=<\v, A\v, \dots, A^{m-1}\v>.$$
{In what follows we will assume $\dim \mathcal{K}_m(A,\v)=m$. }

\IncMargin{1em}
\begin{algorithm}[H]
	\LinesNumbered
	
	\SetKwData{Left}{left}\SetKwData{This}{this}\SetKwData{Up}{up}
	
	\SetKwFunction{Union}{Union}\SetKwFunction{FindCompress}{FindCompress}
	
	\SetKwInOut{Input}{input}\SetKwInOut{Output}{output}
	
	\KwData{$A$, $\v_1:=\v/\|\v\|_2$;}
	
	\While{$j \leq m$ }{
		Compute $\w:=A\v_j$ \;
		\While {$i \leq j$}{
			Compute $h_{i,j}=(\w,\v_i)$ \;
			Compute $\w:=\w-h_{i,j} \v_i$ \;
		}
		Compute $h_{j+1,j}:=\|\w\|_2$ and $\v_{j+1}:=\w/h_{j+1,j}$ \;	
	}
	\caption{Arnoldi Algorithm}{\label{ar}}
\end{algorithm}\DecMargin{1em}

\noindent	The above algorithm produces an orthonormal basis $V_m=[\v_1, \dots, \v_m]$ of the Krylov subspace $K_m(A,\v)$ such that 
$$AV_m=V_mH_m+h_{m+1,m}\v_{m+1}\e_m^T,$$ where the matrix $H_m$ denotes the $m \times m$
upper Hessenberg matrix whose coefficients are the $h_{i,j}$ computed by the algorithm. From the above observations we obtain
\begin{equation} \label{Hmdef}
V_m^TAV_m=H_m.
\end{equation}

\noindent 	Moreover, the following lemma holds :
\begin{lemma}[\cite{Saad92}] \label{saadlem}
	Let $A$ be a $n \times n$ real matrix and $V_m$, $H_m$ the results of m steps of the Arnoldi or Lanczos method applied to A. Then for any polynomial $p_j$ of degree $j \leq m-1$ the following equality holds:
	\begin{equation}
	p_j(A)\v_1=V_mp_j(H_m)\e_1.
	\end{equation}
\end{lemma}

\begin{theorem} \label{directionpreservingprojection} 
	Let $A \in \mathbb{R}^{n \times n}$ be a symmetric matrix.
	For every fixed integer $m$ and $1 \leq m \leq n$ and for any $\s \in \mathbb{R}^{n}$ there exists an orthogonal matrix $L \in \mathbb{R}^{n \times n}$ such that if  $\mathcal{L}= \sd L$ and $\mathcal{L}_A$ is the best approximation in Frobenius norm of $A$ in $\mathcal{L}$, then
	\begin{equation} \label{thesis:directionpreserving}
	p_j(\mathcal{L}_A)\s = p_j(A)\s 
	\end{equation} 
	for any polynomial $p_j$ of degree $j \leq m-1$.
	{Moreover, the thesis is satisfied also by any other orthogonal matrix having, among its columns, $m$ particular columns of $L$ (see \eqref{eq:directionequality}).}
\end{theorem}
\begin{proof} Consider the matrices $V_m$ and $H_m$ constructed from Arnoldi Algorithm applied to $\mathcal{K}_m(A,\s)$ (observe that the first column of $V_m$ is $\v_1:=\s/\|\s\|$). From Lemma~\ref{saadlem} with $j=1$ we have 
	$$A\v_1=V_mH_mV_m^T\v_1.$$ 
	From \eqref{Hmdef} we can write
	
	\begin{equation} \label{eq:action_compression}
	A\v_1=V_mQQ^TV_m^TAV_mQQ^TV_m^T\v_1
	\end{equation}
	for any orthogonal matrix $Q$. In particular, being $V_m^TAV_m$ symmetric,  we can choose in \eqref{eq:action_compression} $Q$ as the orthogonal matrix which diagonalizes $V_m^TAV_m$, i.e.
	\begin{equation}
	A\v_1=V_mQ \begin{bmatrix}
	x_1 & 0 & \dots & 0 \\
	0 & \ddots & \ddots & 0\\
	0 & \ddots & \ddots & 0 \\
	0 & \dots & 0 & x_m
	\end{bmatrix}Q^TV^T_m\v_1,
	\end{equation}
	where $x_i=\e_i^TQ^TV_m^TAV_mQ\e_i$ for $i=1, \dots, m$.
	Consider now the matrix
	\begin{equation} \label{eq:directionequality}
	L=[V_mQ\e_1|\dots|V_mQ\e_m|\q_{m+1}|\dots|\q_{n}]
	\end{equation} 
	where
	$\{\q_{m+1},\dots,\q_{n}\}$ is any orthonormal basis for
	
	\begin{equation} \label{eq:othogonaleq}
	<V_mQ\e_1,\dots,V_mQ\e_m>^{\bot}=<V_m\e_1,\dots,V_m\e_m>^{\bot}
	\end{equation}
	(for example $L$ can be obtained as the product of $m$ Householder matrices, see Lemma~\ref{lem:householder_fixed_columns} in the Appendix),
	set $\mathcal{L} = \sd L$ and consider $\mathcal{L}_A$ the best approximation in Frobenius norm of $A$ in $\mathcal{L}$.
	
	\noindent In order to prove that $\mathcal{L}_A$ satisfies \eqref{thesis:directionpreserving} it is sufficient to prove that 
	\begin{equation} \label{polynomialpreserver}
	\mathcal{L}_A^{j}\v_1=A^{j}\v_1 \hbox{ for } 0 \leq j \leq m-1.
	\end{equation}
	Of course, \eqref{polynomialpreserver} is true for $j=0$. 
	{The equality $\mathcal{L}_A \v_1=A\v_1$ follows observing that {using \eqref{eq:proj_expressionref}} we have
		\begin{equation}
		\begin{split}
		& \mathcal{L}_A\v_1=(\sum_i^n(L^TAL)_{ii}L\e_i(L\e_i)^T)\v_1\\
		&=(\sum_i^{m} x_i(V_mQ\e_i)(V_mQ\e_i)^T)\v_1=A\v_1
		\end{split}
		\end{equation} 
		where in the second equality we take into account that $\q_i^T\v_1=0$ for $i \in \{m+1, \dots, n\}$ (see \eqref{eq:othogonaleq}) and \eqref{eq:directionequality}.
	}	
	
	{	
		Suppose now \eqref{polynomialpreserver} true for all indexes $j$ in $[1,k],\; k \leq m-2$ and let us prove it for $j=k+1$.
		From inductive hypothesis and Lemma~\ref{saadlem} we have
		$$\mathcal{L}_{A}^{k+1} \v_1=\mathcal{L}_{A}\mathcal{L}_{A}^{k} \v_1=\mathcal{L}_{A}A^{k}\v_1=\mathcal{L}_{A}V_mH_m^{k}\e_1.$$
		From direct computation using \eqref{eq:othogonaleq} and the definition of $Q$,  we have $\mathcal{L}_AV_m=V_mH_m$ and thus
		$$\mathcal{L}_{A}V_mH_m^{k}\e_1=V_mH_m^{k+1} \e_1 =A^{k+1}\v_1,$$ 
		where the last equality follows using again Lemma~\ref{saadlem}. Hence \eqref{polynomialpreserver} holds also for $j\in [1,k+1]$.}	
	
\end{proof}

\begin{corollary}\label{remark:ontotallynonlinear}
	{Solutions $U$ of Problem \ref{TNLP} are obtained by using Theorem~\ref{directionpreservingprojection} for $m=2$ and $j=1$. Observe that just two of the columns of such orthogonal matrices $U$ are uniquely determined (they are suitable linear combinations of the vectors $\s$ and $A\s$), and hence one of such $U$ can be chosen as the product of two Householder matrices that can be determined by performing two products of $A$ by a vector plus $O(n)$ FLOPs.}
\end{corollary}
\begin{proof}
	For the second statement see \eqref{eq:directionequality} in the proof of Theorem~\ref{directionpreservingprojection} and Lemma~\ref{lem:householder_fixed_columns} in the Appendix.
\end{proof}
\subsection{{Convergent $\mathcal{L}^{(k)}QN$ scheme}}
{In order to impose \eqref{eq:convergenceconditionprojection} for each $k$, an adaptive choice of the space $\mathcal{L}=\sd U$ is necessary. Any method obtained in this way will be called $\mathcal{L}^{(k)}QN$ extending  the notation $\mathcal{L}QN$ introduced in \cite{DFLZ} to denote the $BFGS$-type methods with $\tilde{B}_k=\mathcal{L}_{{B}_k}$ being $\mathcal{L}$ fixed. As a result of what discussed in Section \ref{sectionsconvergence} and in the first part of this section we report here the following Algorithm \ref{convergentlkqn_no_fq} which can be considered a refinement and an extension of the scheme proposed in \cite{CDTZ}:}

\IncMargin{1em}
\begin{algorithm}[H]
	\LinesNumbered
	
	\SetKwData{Left}{left}\SetKwData{This}{this}\SetKwData{Up}{up}
	
	\SetKwFunction{Union}{Union}\SetKwFunction{FindCompress}{FindCompress}
	
	\SetKwInOut{Input}{input}\SetKwInOut{Output}{output}
	
	\KwData{$\x_0\,\in \mathbb{R}^n, B_0$ pd, $\d_0=-\g_0$, $k=0$; \\}
	
	\While{$\g_k\neq 0$}{
		$\x_{k+1}=\x_k+\lambda_k \d_k$ \tcc*[r]{$\lambda_k$ verifies conditions \eqref{AG1}, \eqref{AG2}}
		$\s_k=\x_{k+1}-\x_k $\;
		$\y_k=\g_{k+1}-\g_k $\;
		
		\tcc{\footnotesize{Defininition of the new algebra $\mathcal{L}^{(k)}$}} 
		\eIf{$\|B_k\s_k-\frac{\s_k B_k \s_k}{\|\s_k\|^2}\s_k\| < toll$}
		{define $U_k$ applying Lemma~\ref{lemma:eigenvector_tnlp_solution}  \nllabel{algline:1if_no_qt}\;}
		{define $U_k$ applying Corollary \ref{remark:ontotallynonlinear}  \nllabel{algline:u_kconstruction_no_qt}\;}
		Compute $\mathcal{L}^{(k)}_{B_k} $ \nllabel{algline:B_kcomputation_no_qt} \;
		\tcc{\footnotesize{$\mathcal{L}^{(k)}:= \sd U_k$ satisfies $\mathcal{L}^{(k)}_{{B}_k}\s_k=B_k\s_k$ }} 
		$B_{k+1}= \Phi(\mathcal{L}^{(k)}_{{B}_k}, \s_k,\y_k, \phi)$ \nllabel{algline:B_k_update} \;
		Compute $\d_{k+1}=-B_{k+1}^{-1}\g_{k+1}$ \nllabel{algline:sm_no_qt}\;
		Set $k:=k+1$ \;		
	}
	\caption{{A convergent $\mathcal{L}^{(k)}$QN}}{\label{convergentlkqn_no_fq}}
\end{algorithm}\DecMargin{1em}
\noindent  In more details, observe that, to perform  Line \ref{algline:u_kconstruction_no_qt} of  Algorithm~\ref{convergentlkqn_no_fq}, it is necessary to 
apply Corollary \ref{remark:ontotallynonlinear}  to $B_k=\Phi(\mathcal{L}^{(k-1)}_{{B}_{k-1}}, \s_{k-1},\y_{k-1}, \phi)$ and $\s_k$, obtaining $U_k:=\mathcal{H}(\h^{(k)}_2)\mathcal{H}(\h^{(k)}_1)$. The vectors $\h_1^{(k)}$ and $\h_2^{(k)}$ can be determined by performing two products of $B_k$ by a vector. As $B_k$ is a low rank correction of the low complexity matrix $\mathcal{L}_{B_{k-1}}^{(k-1)}$, such products can be calculated in $O(n)$ FLOPs (see Corollary \ref{remark:ontotallynonlinear}).
To compute $\mathcal{L}_{B_{k}}^{(k)}$ in Line \ref{algline:B_kcomputation_no_qt}, observe that, by Lemma~\ref{repspecprop}, 
\begin{equation*}
\mathcal{L}^{(k)}_{B_{k}}=\mathcal{L}^{(k)}_{\mathcal{L}_{{B}_{k-1}}^{(k-1)}}- \mathcal{L}^{(k)}_{\frac{\mathcal{L}^{(k-1)}_{{B}_{k-1}}\s_{k-1}\s_{k-1}^T\mathcal{L}_{{B}_{k-1}}^{({k-1})}}{\s_{k-1}^T\mathcal{L}_{{B}_{k-1}}^{({k-1})}\s_{k-1}}}+\mathcal{L}^{(k)}_{\frac{\y_{k-1}\y_{k-1}^T}{\y_{k-1}^T\s_{k-1}}}+(\phi\; \s_{k-1}^T\mathcal{L}^{({k-1})}_{{B}_{k-1}}\s_{k-1})\mathcal{L}^{(k)}_{\v_{k-1}\v_{k-1}^T},
\end{equation*}
and hence, it is sufficient to compute its eigenvalues (see \eqref{eq:proj_expressionref}), i.e.,

\begin{equation} \label{eq:update_eigenvalues}
\begin{split}
\lambda(\mathcal{L}_{{B}_{k}}^{({k})})&=d([U_{k}^T{B}_{k}U_{k}]) \\
&=d(U_{k}^T\mathcal{L}_{{B}_{k-1}}^{(k-1)}U_{k})-d(U_{k}^T\frac{\mathcal{L}^{(k-1)}_{{B}_{k-1}}\s_{k-1}\s_{k-1}^T\mathcal{L}_{{B}_{k-1}}^{({k-1})}}{\s_{k-1}^T\mathcal{L}_{{B}_{k-1}}^{({k-1})}\s_{k-1}}U_{k})+\\
& +d(U_{k}^T\frac{\y_{k-1}\y_{k-1}^T}{\y_{k-1}^T\s_{k-1}}U_{k}+{(\phi\; \s_{k-1}^T\mathcal{L}^{({k-1})}_{{B}_{k-1}}\s_{k-1})}U_{k}^T\v_{k-1} \v_{k-1}^T U_{k}).
\end{split}	
\end{equation}

\noindent {Notice that the above equality is an extension of an eigenvalues updating formula obtained in \cite{DFLZ} where $\mathcal{L}^{(k)}\equiv \mathcal{L}$ for all $k$.}
\subsection{Complexity} \label{sec:complexity_no_qt}

For every $k$ the orthogonal matrices at Line~\ref{algline:1if_no_qt} or Line~\ref{algline:u_kconstruction_no_qt} of Algorithm~\ref{convergentlkqn_no_fq} are the product of at most two (only one if Line~\ref{algline:1if_no_qt}) Householder reflections, that can be constructed in $O(n)$ FLOPs (see Lemma~\ref{lem:householder_fixed_columns} in the Appendix). Now, to calculate $\lambda(\mathcal{L}_{B_k}^{(k)})$ in \eqref{eq:update_eigenvalues}, we compute the matrix vector products $\mathcal{L}_{B_{k-1}}^{(k-1)}\s_{k-1}$ in $O(n)$ FLOPs, and the same amount of operations is sufficient to compute $d(U_{k}^T\mathcal{L}_{{B}_{k-1}}^{(k-1)}U_{k})$ (using {Proposition 1 in}  \cite{CDZarnoldipreserver}).
Finally, observe that Line  \ref{algline:sm_no_qt} of Algorithm~\ref{convergentlkqn_no_fq} can be performed using Sherman-Morrison formula, which states that $B_{k+1}^{-1}$ is a low rank correction of $(\mathcal{L}_{B_k}^{(k)})^{-1}$; for example if $\phi=0$ in Line \ref{algline:B_k_update} of Algorithm \ref{convergentlkqn_no_fq}, then 
\begin{equation*}
B_{k+1}^{-1}=(I-\frac{ \s_k \y_k^T}{\s_k^T \y_k}) (\mathcal{L}_{B_{k}}^{(k)})^{-1} (I- \frac{ \y_k \s_k^T}{\s_k^T\y_k})+\frac{ \s_k \s_k^T}{\s_k^T\y_k}.
\end{equation*} 
Thus it is possible to infer that the computational complexity of Algorithm \ref{convergentlkqn_no_fq} is $O(n)$ in space and time (to store the matrices $\mathcal{L}_{{B}_{k}}^{({k})}=U_kd(\z_{B_k})U_k^T$ it is sufficient to store $\z_{B_k}$ and the vectors $\h_i^{(k)}$ needed to define $U_k$). 
When $\phi=0$, assuming that the matrices $U_k$ are always constructed according to Line~\ref{algline:u_kconstruction_no_qt} of Algorithm \ref{convergentlkqn_no_fq}, a straightforward implementation of Algorithm \ref{convergentlkqn_no_fq} requires roughly  $70n$ multiplications and the storage of $15$ vectors of length $n$.
\section{The quadratic finite termination property} \label{sec:finite_quadratic_termination}
In literature Quasi-Newton methods are studied  that terminate in a finite number of steps when applied to quadratic functions (quadratic finite termination). See \cite{kolda1998bfgs,nazareth1979relationship} and references therein.
In this section, extending the analogous result obtained in \cite{kolda1998bfgs} for $L$-$BFGS$, we will introduce conditions on $\tilde{B}_k$ (see \eqref{h0condion}) which endow the $\mathcal{S}$ $BFGS$-type methods with the quadratic finite termination property. 

Let us consider a pd matrix $A$ and the problem 
\begin{equation} \label{prob:quadratic}
\min_{\x \in \mathbb{R}^n}f(\x)\hbox{ where }f(\x):=\frac{1}{2}\x^TA\x-\x^T\b.
\end{equation}
In order to solve Problem \eqref{prob:quadratic} consider the following Algorithm \ref{alg:bfgstquadratic} which is {the $\mathcal{S}$ version of} Algorithm \ref{gBc} where we use the exact line search and where we set $H_k=B_k^{-1}$, $\tilde{H}_k=\tilde{B}_k^{-1}$ and $\phi=0$ (in Line 8 we have the Sherman-Morrison representation of $H_{k+1}=B_{k+1}^{-1}$).

\IncMargin{1em}
\begin{algorithm}[H]
	\LinesNumbered
	
	\SetKwData{Left}{left}\SetKwData{This}{this}\SetKwData{Up}{up}
	
	\SetKwFunction{Union}{Union}\SetKwFunction{FindCompress}{FindCompress}
	
	\SetKwInOut{Input}{input}\SetKwInOut{Output}{output}
	
	\KwData{$\x_0\,\in \mathbb{R}^n, \,\g_0=A\x_0-\b,\, \tilde{H}_0=H_0$ pd, 
		$\d_0=-H_0\g_0$, k=0; \\}
	
	\While{$\g_k\neq 0$}{

		$\x_{k+1}=\x_k+\lambda_k \d_k$ \tcc*[r]{{$\lambda_k:=\arg \min_{\lambda}f(\x_k+\lambda \d_k)$}}
		
		$\s_k=\x_{k+1}-\x_k $\;
		$\g_{k+1}=A\x_{k+1}-\b$\;
		
		$\y_k=\g_{k+1}-\g_k $\;
		$\rho_k=1/\s_k^T\y_k$ \;
		Define $\tilde{H}_k$ pd \;
		$H_{k+1}=(I-\rho_k \s_k \y_k^T) \tilde{H}_k (I-\rho_k \y_k \s_k^T)+\rho_k \s_k \s_k^T $\;
		Set $\d_{k+1}=-H_{k+1}\g_{k+1}$\;
		Set $k:=k+1$ \;		
	}
	\caption{$BFGS$-type for quadratic problems}{\label{alg:bfgstquadratic}}
\end{algorithm}\DecMargin{1em}
\begin{theorem} \label{theo:good_choice_bfgs_type}
	Let us consider Algorithm \ref{alg:bfgstquadratic}. If 	
	\begin{equation} \label{h0condion}
	\tilde{H}_k \g_{k+1}=  \beta_k H_0\g_{k+1} \hbox{ for some } \beta_k \neq 0,
	\end{equation}
	then we have :
	\begin{equation} \label{gradientsearchdirection}
	\g_{k+1}^T\s_j=0 \hbox{ for all } j=0,\dots,k;
	\end{equation}
	\begin{equation} \label{aconjugatesearchdirec}
	\s_{k+1}^TA\s_j=0 \hbox{ for all } j=0,\dots,k;
	\end{equation}
	\begin{equation} \label{spansearchdirec}
	\Span\{\s_0,\dots,\s_{k+1}\}=\Span\{H_0\g_0,\dots,H_0\g_{k+1}\};
	\end{equation}
	
\end{theorem}
\begin{proof}
	By induction. The case $k=0$ can be easily verified. Let us suppose the thesis true for $k=0,\dots, \hat{k}-1$ and prove it for $k=\hat{k}$. Let us prove \eqref{gradientsearchdirection} : $\g_{\hat{k}+1}^T\s_{\hat{k}}=0$ since we are using exact line search; if $j < \hat{k}$ we have \begin{equation}
	\g_{\hat{k}+1}^T\s_{j}=\g_{\hat{k}}^T\s_{j}+\y_{\hat{k}}^T\s_j=\g_{\hat{k}}^T\s_{j}+\s_{\hat{k}}^TA\s_{j}=0
	\end{equation} by induction hypothesis. To prove \eqref{aconjugatesearchdirec} observe that for $j<\hat{k}$ \begin{equation}
	\begin{split}
	& \s_{\hat{k}+1}^TA\s_{j}=  -\lambda_{\hat{k}+1}\g_{\hat{k}+1}^TH_{\hat{k}+1}\y_j= \\
	& -\lambda_{\hat{k}+1}\g_{\hat{k}+1}^T((I-\rho_{\hat{k}} \s_{\hat{k}} \y_{\hat{k}}^T) \tilde{H}_{\hat{k}} (I-\rho_{\hat{k}} \y_{\hat{k}} \s_{\hat{k}}^T)+\rho_{\hat{k}} \s_{\hat{k}} \s_{\hat{k}}^T)\y_j =\\
	&
	-\lambda_{\hat{k}+1}\g_{\hat{k}+1}^T\tilde{H}_{\hat{k}}\y_j=-\beta_{\hat{k}}\lambda_{\hat{k}+1}\g_{\hat{k}+1}^T{H}_{0}\y_j=0 
	\end{split}
	\end{equation}
	\noindent	where the third equality follows observing that $\g_{\hat{k}+1}^T\s_{\hat{k}}=0$ and that $\s_{\hat{k}}^T\y_{j}=0$ for $j<\hat{k}$ by induction hypothesis; the fourth equality follows by \eqref{h0condion}; the last equality follows observing that, since $\g_{\hat{k}+1}^T\s_{\hat{i}}=0 \hbox{ for all } j=0,\dots, \hat{k}$ and $\Span\{\s_0,\dots,\s_{\hat{k}}\}=\Span\{H_0\g_0,\dots,H_0\g_{\hat{k}}\}$ by induction hypothesis, it holds that
	\begin{equation} \label{gradientH0}
	\g_{\hat{k}+1}^TH_0\g_{j}=0 \hbox{ for all } j=0, \dots, \hat{k}.
	\end{equation}
	{Now let us consider the case $j=\hat{k}$. Since $\s_{\hat{k}+1}=-\lambda_{\hat{k}+1}H_{\hat{k}+1}\g_{\hat{k}+1}$, by direct computation using the definition of $H_{\hat{k}+1}$,} it can be verified that $\s_{\hat{k}+1}^TA\s_{\hat{k}}=\s_{\hat{k}+1}^T\y_{\hat{k}}=0.$
	{Let us prove now \eqref{spansearchdirec} : we have \begin{equation}
		\begin{split}
		&\s_{\hat{k}+1}=-\lambda_{\hat{k}+1}H_{\hat{k}+1}\g_{\hat{k}+1}=-\lambda_{\hat{k}+1}\tilde{H}_{\hat{k}} \g_{\hat{k}+1}+ \lambda_{\hat{k}+1}\rho_{\hat{k}}\y_{\hat{k}}^T\tilde{H}_{\hat{k}}\g_{\hat{k}+1} \s_{\hat{k}} = \\
		&-\beta_{\hat{k}}\lambda_{\hat{k}+1}{H}_{0}\g_{\hat{k}+1}+ \lambda_{\hat{k}+1}\rho_{\hat{k}}\y_{\hat{k}}^T\tilde{H}_{\hat{k}}\g_{\hat{k}+1} \s_{\hat{k}} 
		\end{split}
		\end{equation}} and hence
	{$$\Span\{H_0\g_0,\dots,H_0\g_{\hat{k}+1}\} = \Span\{\s_0,\dots,\s_{\hat{k}+1}\}$$
		since $\Span\{H_0\g_0,\dots,H_0\g_{\hat{k}}\} = \Span\{\s_0,\dots,\s_{\hat{k}}\}$
		and $\{\s_0,\dots,\s_{\hat{k}+1}\}$ are linearly independent since they are $A$-conjugate. }
\end{proof}

\begin{corollary}
	If {the pd matrices}  $\tilde{H}_k$ satisfy hypothesis of Theorem~\ref{theo:good_choice_bfgs_type}, then Algorithm \ref{alg:bfgstquadratic} generates the same iterates as the Conjugate Gradient method preconditioned with $H_0$ and hence it converges in at most $n$ steps. 
\end{corollary}
\begin{proof}
	Analogous to the proof of Corollary 2.3 in \cite{kolda1998bfgs}, observing that under hypothesis of Theorem~\ref{theo:good_choice_bfgs_type} conditions   \eqref{gradientsearchdirection}, \eqref{aconjugatesearchdirec} and \eqref{spansearchdirec} hold for Algorithm \ref{alg:bfgstquadratic}.
\end{proof}

\noindent Interestingly enough, using the above corollary it can be shown that the iterates of Algorithm \ref{alg:bfgstquadratic} coincide with those from $BFGS$ and $L$-$BFGS$ since they all coincide with the Preconditioned Conjugate Gradient (see \cite{nazareth1979relationship,kolda1998bfgs}). \\
We can now prove that the convergence condition \eqref{directionpreserving} and the quadratic termination condition \eqref{h0condion} can be verified simultaneously if $\tilde{B}_k^{-1}=\tilde{H}_k=\mathcal{L}_{B_k}^{-1}$ {provided that $H_0$ in  \eqref{h0condion} is a multiple of the identity}.

\begin{lemma} \label{lem:threehouseholder}
	{For any pair of vectors $\s_k$, $\g_{k+1}$ and pd matrix $B_k$ generated by Algorithm \ref{alg:bfgstquadratic} with $H_0=I$, there exists a low complexity orthogonal matrix ${L}_{k}$ and hence a matrix algebra $\mathcal{L}^{(k)}= \sd L_{k}$ such that}
	\begin{equation}\label{eq:quadratic_termination_condition}
	\begin{split}
	\mathcal{L}^{(k)}_{B_{k}}\s_{k}&={B_{k}}\s_{k}, \\
	\mathcal{L}^{(k)}_{B_{k}}\g_{k+1}&=\alpha_k\g_{k+1} \hbox{ for some } \alpha_k \neq 0.
	\end{split}
	\end{equation}
	\noindent $L_k$ can be effectively constructed as the product of at most three Householder matrices.	
\end{lemma}
\begin{proof}{For the sake of simplicity we use, in the following, the symbols $L$ and $\mathcal{L}$ in place of $L_{k}$ and $\mathcal{L}^{(k)}$.} 
	\begin{enumerate}
		\item Case $B_k\s_k = \gamma \s_k$. \\ 	From Theorem~\ref{theo:good_choice_bfgs_type} we have $\g_{k+1}^T\s_k=0$. Any orthogonal matrix $L$ which has among its columns $\s_k/\|\s_k\|$ and $\g_{k+1}/\|\g_{k+1}\|$ is such that, defining $\mathcal{L}=\sd L$, $\mathcal{L}_{B_k}$ satisfies conditions in \eqref{eq:quadratic_termination_condition} (the columns of $L$ are eigenvectors of any matrix in $\mathcal{L}$). 
		One of such orthogonal matrix $L$ can be constructed as the product of two orthogonal Householder matrices ({see Lemma~\ref{lem:householder_fixed_columns} in Appendix and see \cite{CDZarnoldipreserver} for more details}).
		\item {Case $B_k\s_k \neq \gamma \s_k$.\\} 
		Any matrix $L$ in \eqref{eq:directionequality} with $m=2$ satisfies $\mathcal{L}_{B_k}\s_k=B_k\s_k$ if $\mathcal{L}=\sd L$; it is then enough to consider a  matrix $L$  where $\g_{k+1}/\|\g_{k+1}\|$ is chosen to be one of the vectors $\q_i$; observe that this can be done since, from Theorem~\ref{theo:good_choice_bfgs_type},  $\g_{k+1}^T\s_k=0=\g_{k+1}^T\g_k$ (see \eqref{gradientH0} with $H_0=I$) and {since} the first two columns of $L$ in \eqref{eq:directionequality}, namely $V_2Q\e_1 \hbox{ and } V_2Q\e_2$, are suitable linear combinations of $\s_k$ and $B_k\s_k=-\lambda_k \g_k$ {(see the proof of Theorem~\ref{directionpreservingprojection} with $m=2$ and $\s_k$, $B_k$ in the roles of $\s$ and $A$ respectively)}.  
		{An orthogonal matrix ${L}$ with three columns fixed as $V_2Q\e_1, \; V_2Q\e_2$ and $\g_{k+1}/\|\g_{k+1}\|$, can be constructed as the product of three orthogonal Householder matrices ({see Lemma~\ref{lem:householder_fixed_columns} in Appendix and see \cite{CDZarnoldipreserver} for more details}).}%
	\end{enumerate}	
\end{proof} 
\section{A convergent $\mathcal{L}^{(k)}$QN method with quadratic termination property} \label{sec:theproposal}
The $\mathcal{L}^{(k)}QN$ scheme that we consider {in this section} combines the results obtained in Section \ref{sectionsconvergence} for the $\mathcal{S}$ecant scheme with $\phi=0$ and in Section \ref{sec:finite_quadratic_termination} for quadratic termination, setting in both $\tilde{B}_k=\mathcal{L}^{(k)}_{{B}_k}$. In particular it combines the convergence result stated in Theorem~\ref{Sconvergence} for general non linear problems with the  quadratic termination result obtained in Theorem~\ref{theo:good_choice_bfgs_type}. The main motivation for this choice can be traced {to the key observation that in this way the resulting method coincides}, as already pointed out in Section \ref{sec:finite_quadratic_termination}, with $BFGS$ and $L$-$BFGS$ when applied on quadratic problems using exact line search. 
{\subsection{The proposed method}}
\IncMargin{1em}
\begin{algorithm}[H]
	\LinesNumbered
	
	\SetKwData{Left}{left}\SetKwData{This}{this}\SetKwData{Up}{up}
	
	\SetKwFunction{Union}{Union}\SetKwFunction{FindCompress}{FindCompress}
	
	\SetKwInOut{Input}{input}\SetKwInOut{Output}{output}
	
	\KwData{$\x_0\,\in \mathbb{R}^n, B_0=I$, $toll$,
		$\d_0=-\g_0$, $k=0$; \\}
	
	\While{$\g_k\neq 0$}{

		$\x_{k+1}=\x_k+\lambda_k \d_k$ \tcc*[r]{$\lambda_k$ verifies conditions \eqref{AG1}, \eqref{AG2}}
		$\s_k=\x_{k+1}-\x_k $\;
		$\y_k=\g_{k+1}-\g_k $\;
		\tcc{\footnotesize{Defininition of the new algebra $\mathcal{L}^{(k)}$}} 
		
		{\eIf{$\|B_k\s_k-\frac{\s_k B_k \s_k}{\|\s_k\|^2}\s_k\| < toll$ \nllabel{algline:if_cond}}{
				Define $\overline{\g}_{k+1}$ as the projection of $\g_{k+1}$ on  $<\s_k>^\perp$ \; 
				Define $U_k$ using Case 1. in Lemma~\ref{lem:threehouseholder} \nllabel{algline:1if}  \;
			}
			{Define $\overline{\g}_{k+1}$ as the projection of $\g_{k+1}$ on  $<\s_k,B_k\s_k>^\perp$ \; 
				Define $U_k$ using Case 2. in Lemma~\ref{lem:threehouseholder} \nllabel{algline:generic_cond}  \;
		}}
		Compute $\mathcal{L}^{(k)}_{B_k} $ \nllabel{algline:compute_proj2} \;
		\tcc{\footnotesize{$\mathcal{L}^{(k)}:= \sd U_k$ verifies $\mathcal{L}^{(k)}_{{B}_k}\s_k=B_k\s_k$ and $\mathcal{L}^{(k)}_{{B}_k} \overline{\g}_{k+1}=\alpha_k \overline{\g}_{k+1}$}}	
		$B_{k+1}= \Phi(\mathcal{L}^{(k)}_{{B}_k}, \s_k,\y_k, 0)$ \;
		Compute $\d_{k+1}=-B_{k+1}^{-1}\g_{k+1}$\;
		Set $k:=k+1$ \;		
	}
	\caption{{A convergent $\mathcal{L}^{(k)}$QN method with quadratic termination property verified if exact line search is used.}}{\label{convergentlkqn}}
\end{algorithm}\DecMargin{1em}

\noindent Observe that the applicability of Lemma~\ref{lem:threehouseholder}, and hence the existence of the orthogonal matrices $U_k$ at lines \ref{algline:1if} and \ref{algline:generic_cond} of Algorithm \ref{convergentlkqn}, are guaranteed by the definition of $\overline{\g}_{k+1}$. Indeed, in Lemma~\ref{lem:threehouseholder}, where $f$ is quadratic, $\g_{k+1}$ is orthogonal to $\s_k$ and to $B_k\s_k$. When $f$ is not quadratic, $\g_{k+1}$ has to be replaced by the vector $\overline{\g}_{k+1}$ which is, by construction, orthogonal to both $\s_k$ and $B_k\s_k$.
\noindent  In particular, to perform  Line \ref{algline:generic_cond}  of  Algorithm~\ref{convergentlkqn}, one computes the projection of ${\g}_{k+1}$ on the space $<\s_k, B_k\s_k>^{\perp}$, that is, $\overline{\g}_{k+1}:=(I-VV^T)\g_{k+1}$ being $V:=[\v_1|\v_2]$ an orthonormal basis of $<\s_k,B_k\s_k>$, and then   
apply Lemma~\ref{lem:threehouseholder} to $B_k=\Phi(\mathcal{L}^{(k-1)}_{{B}_{k-1}}, \s_{k-1},\y_{k-1}, 0)$, $\s_k$ and $\overline{\g}_{k+1}$, to obtain $U_k:=\mathcal{H}(\h^{(k)}_3)\mathcal{H}(\h^{(k)}_2)\mathcal{H}(\h^{(k)}_1)$ (see, moreover, Lemma~\ref{lem:householder_fixed_columns} in the Appendix). For Line \ref{algline:1if}, proceed analogously; in this case $U_k:=\mathcal{H}(\h^{(k)}_2)\mathcal{H}(\h^{(k)}_1)$. Regarding {Line \ref{algline:compute_proj2}} observe that, as in Algorithm \ref{convergentlkqn_no_fq}, they consist in computing the eigenvalues of $\mathcal{L}_{B_k}^{(k)}$ by \eqref{eq:update_eigenvalues}.
\subsection{{Complexity}}
{An analogous analysis as in Section \ref{sec:complexity_no_qt} permits to infer that the computational complexity of Algorithm \ref{convergentlkqn} is $O(n)$ in space and time.}
Assuming that the matrices $U_k$ are always constructed according to Line~\ref{algline:generic_cond} of Algorithm \ref{convergentlkqn}, a straightforward implementation of Algorithm \ref{convergentlkqn} requires roughly  $120n$ multiplications and the storage of $17$ vectors of length $n$.
\section{Numerical Results} \label{sec:numerical_results}
In our numerical experimentation we have used performance profiles (see \cite{dolan2002benchmarking}) in order to investigate and compare the numerical behavior of {Algorithm \ref{convergentlkqn_no_fq} with $\phi=0$ (refinement of the method introduced in \cite{CDTZ}), Algorithm \ref{convergentlkqn},  $\mathcal{D}QN$ \cite{CCD}, $\mathcal{H}QN$ \cite{BDCFZ,DFLZ} and $L$-$BFGS$ with $M=5$ and $M=30$ \cite{LN}{. The latter method, that has been implemented by the Poblano toolbox \cite{dunlavy2010poblano}, has a computational cost of roughly $4Mn$ multiplications and requires the storage of $(4M+2$ vectors to be implemented. We have tested the algorithms on a set of medium/large scale problems  using the line-search routine provided in Poblano}, i.e., the Mor\'e-Thuente cubic interpolation line search (which implements the Strong-Wolfe conditions) {enforcing the reproducibility of our results}. In order to make a fair comparison we have used for all the algorithms the same stopping criteria as those from Poblano. The results have been obtained on a laptop running Linux with 16Gb memory and CPU Intel(R) Core(TM) {i7-8th generation} CPU with clock 2.00GHz. The scalar code is written and executed in {MATLAB R2018b}. We have used the following parameters where the names of the variables are the same as those from Poblano {(LineSearch\_ftol=$\alpha$ in \eqref{AG1}  and LineSearch\_gtol=$\beta$ in \eqref{AG2})} :
	
	\begin{lstlisting}[style=Matlab-editor]
	LineSearch_xtol=1e-15;
	LineSearch_ftol=1e-4; 
	LineSearch_gtol=0.9; 
	LineSearch_stpmin=1e-15;
	LineSearch_stpmax=1e15;
	LineSearch_maxfev=20;
	StopTol=1e-6;
	MaxIters=10000;
	MaxFuncEvals=50000;
	RelFuncTol=1e-20.
	\end{lstlisting}
	
	\noindent Finally,	let us point out that, as in Poblano, the successful termination is achieved when $\|g_k\|_2/n\leq StopTol$ being $n$ the dimension of the problem.  
	
	{ In all the following  Figures   ``$\mathcal{L}^{(k)}$QN Sc'' and ``$\mathcal{L}^{(k)}$QN'' indicate Algorithm \ref{convergentlkqn_no_fq} using, respectively, scaling as in Remark \ref{remark:scaling} or not. Analogously, ``$\mathcal{L}^{(k)}$QN(q.t.)~Sc'' and ``$\mathcal{L}^{(k)}$QN(q.t.) '' indicate Algorithm~\ref{convergentlkqn} using, respectively, scaling as in Remark \ref{remark:scaling} or not.}
	
	\subsubsection{Experiment 1}
	In this experiment we have chosen a problem set from CUTEst \cite{gould2015cutest} where $L$-$BFGS$ performs poorly. See Table \ref{problem_set} for the complete list of considered problems. In Figure \ref{fig:experiment1}  we show, using a logarithmic scale, the performance profiles of the selected solvers. 
	
	\begin{figure}[ht!]
		\centering
		{\includegraphics[width=\columnwidth]{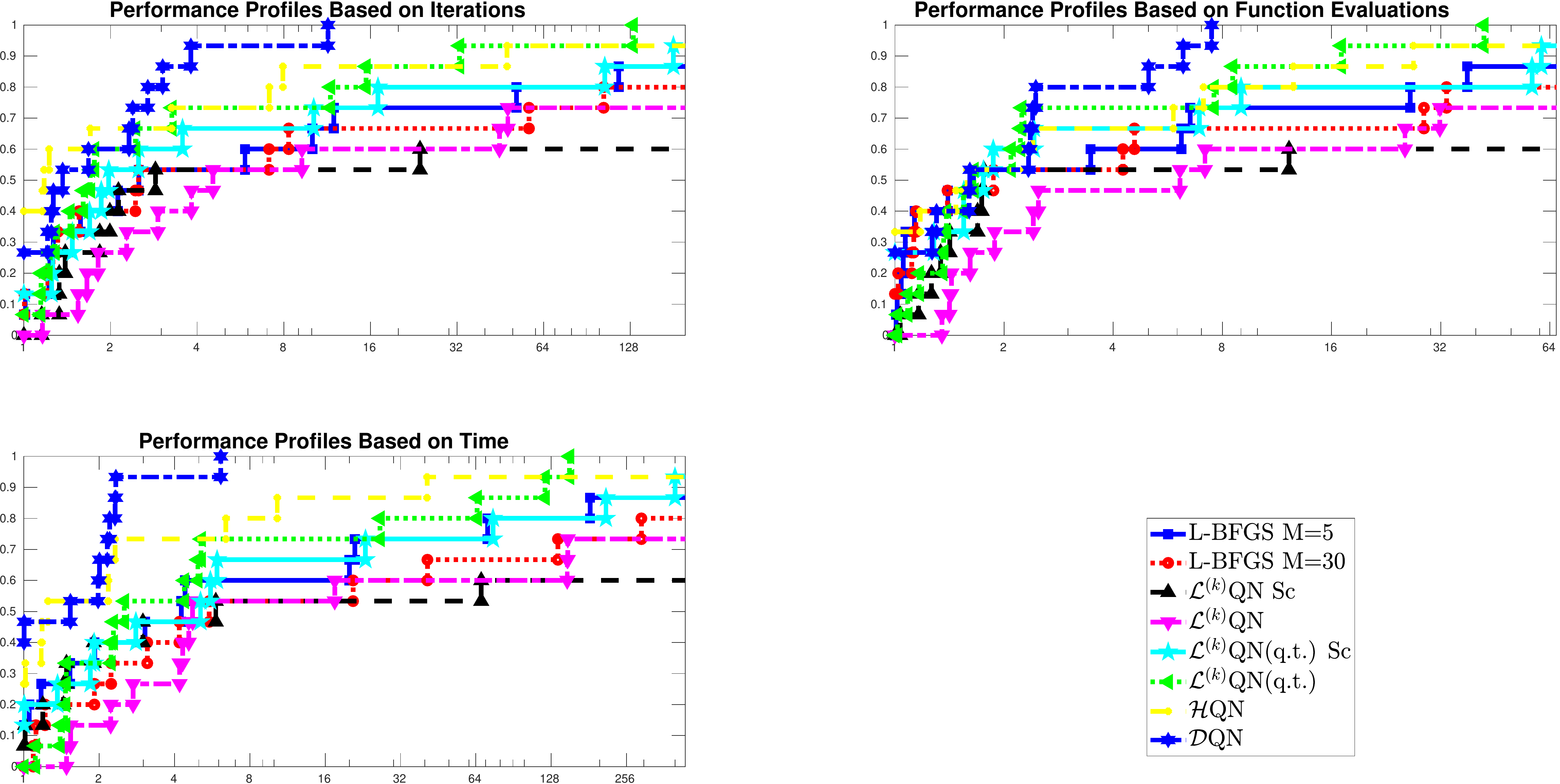}}
		\caption{Performance profiles for Algorithm \ref{convergentlkqn_no_fq}, Algorithm \ref{convergentlkqn}, $\mathcal{D}QN$ \cite{CCD}, $\mathcal{H}QN$ \cite{BDCFZ,DFLZ} and $L$-$BFGS$ with $M=5$ and $M=30$ \cite{LN} on a set of 14  problems from CUTEst \cite{gould2015cutest}. LineSearch\_ftol=1e-4; LineSearch\_gtol=0.9;}\label{fig:experiment1}
	\end{figure}

	\begin{table}[htb]
		\centering
		\caption{Problem Set}
		\label{problem_set}
		\begin{tabular}{lll}
			\textbf{Prob} & \textbf{Dim.} & \textbf{N.Z.}\\  
			1] BROYDN7D & 5000 & 17497 \\  
			2] CHAINWOO & 10000 & 19999 \\  
			3] CURLY10 & 1000 & 10945 \\ 
			4] EIGENBLS & 2550 & 3252525 \\ 
			5] EIGENCLS & 2652 & 3517878 \\  
			6] GENHUMPS & 5000 & 9999 \\ 
			7] GENROSE & 500 & 999 \\ 
		\end{tabular}\begin{tabular}{lll}
			\textbf{Prob} & \textbf{Dim.} & \textbf{N.Z.}\\  
			8] MODBEALE & 20000 & 39999 \\ 
			9] MSQRTALS & 4900 & 12007450 \\ 
			10] MSQRTBLS & 4900 & 12007450 \\ 
			11] NONCVXU2 & 10000 & 39987 \\ 
			12] SBRYND & 1000 & 6979 \\ 
			13] TESTQUAD & 1000 & 1000 \\ 
			14] TRIDIA & 5000 & 9999 \\  
		\end{tabular}
	\end{table}

	\subsubsection{Experiment 2}
	In this experiment we have investigated the problem of approximating a given matrix $A \in A^{m \times n}$ by a rank-$k$ approximation of the form $UV^T$, i.e., the function we wish to optimize is 
	\begin{equation} \label{eq:rank_problem}
	\min _{U \in \mathbb{R}^{m \times k}, V \in \mathbb{R}^{n \times k}}\|A-UV^T\|_F^2.
	\end{equation}
	
	\noindent Problem \eqref{eq:rank_problem} arises in may applications (see for example \cite{elden2006numerical} for applications connected with data mining). 
	In particular, we focus on the dimensionality reduction problem ($(m+n)k << mn$) for MINST database \cite{lecun}. The MINST test-set contains $10000$ labeled handwritten digits from $0$ to $9$ stored as $28 \times 28$ matrices. For each class, we solve problem \eqref{eq:rank_problem} where $A$ is a $m \times n = 28^2  \times \, class$-$size$, being $class$-$size$ the number of examples contained in the dataset for the considered digit. In Figures \ref{fig:experiment3} and \ref{fig:experiment4}  we show, using a logarithmic scale, the performance profile of the selected solvers when $k=2^6$ and $k=2^7$. For details on the choice of the parameters see the preliminaries of this section; we use as $\x_0$ a random vector. In Table \ref{minst_dimensions} we report the dimensions of the involved problems.

	\begin{figure}[ht!]
		\centering
		{\includegraphics[width=\columnwidth]{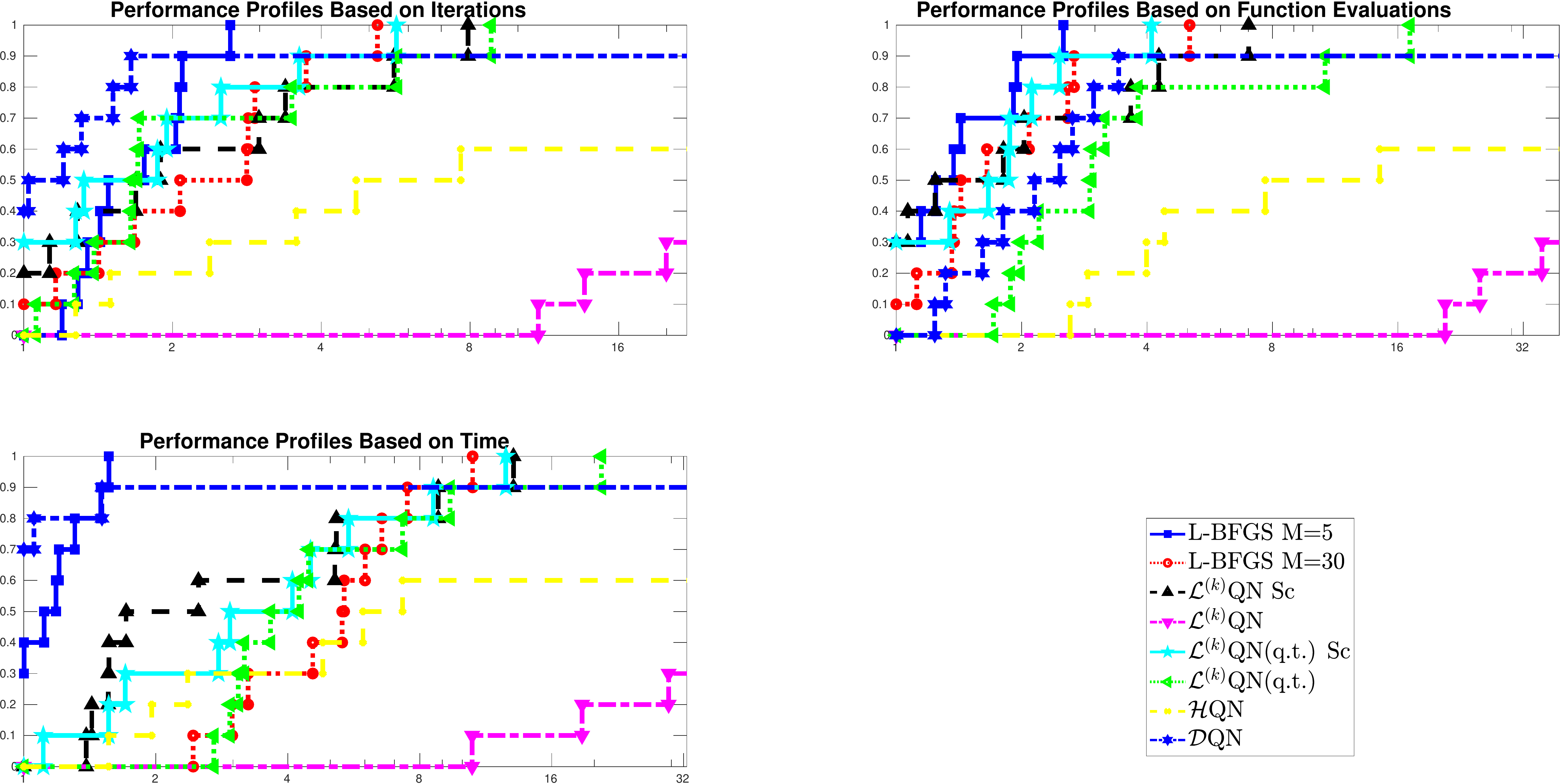}} 
		\caption{Performance profiles for Algorithm \ref{convergentlkqn_no_fq}, Algorithm \ref{convergentlkqn}, $\mathcal{D}QN$ \cite{CCD}, $\mathcal{H}QN$ \cite{BDCFZ,DFLZ} and $L$-$BFGS$ with $M=5$ and $M=30$ \cite{LN} when $k=2^6$. LineSearch\_ftol=1e-4; LineSearch\_gtol=0.9.}\label{fig:experiment3}
	\end{figure}
	
	\begin{figure}[ht!]
		\centering
		{\includegraphics[width=\columnwidth]{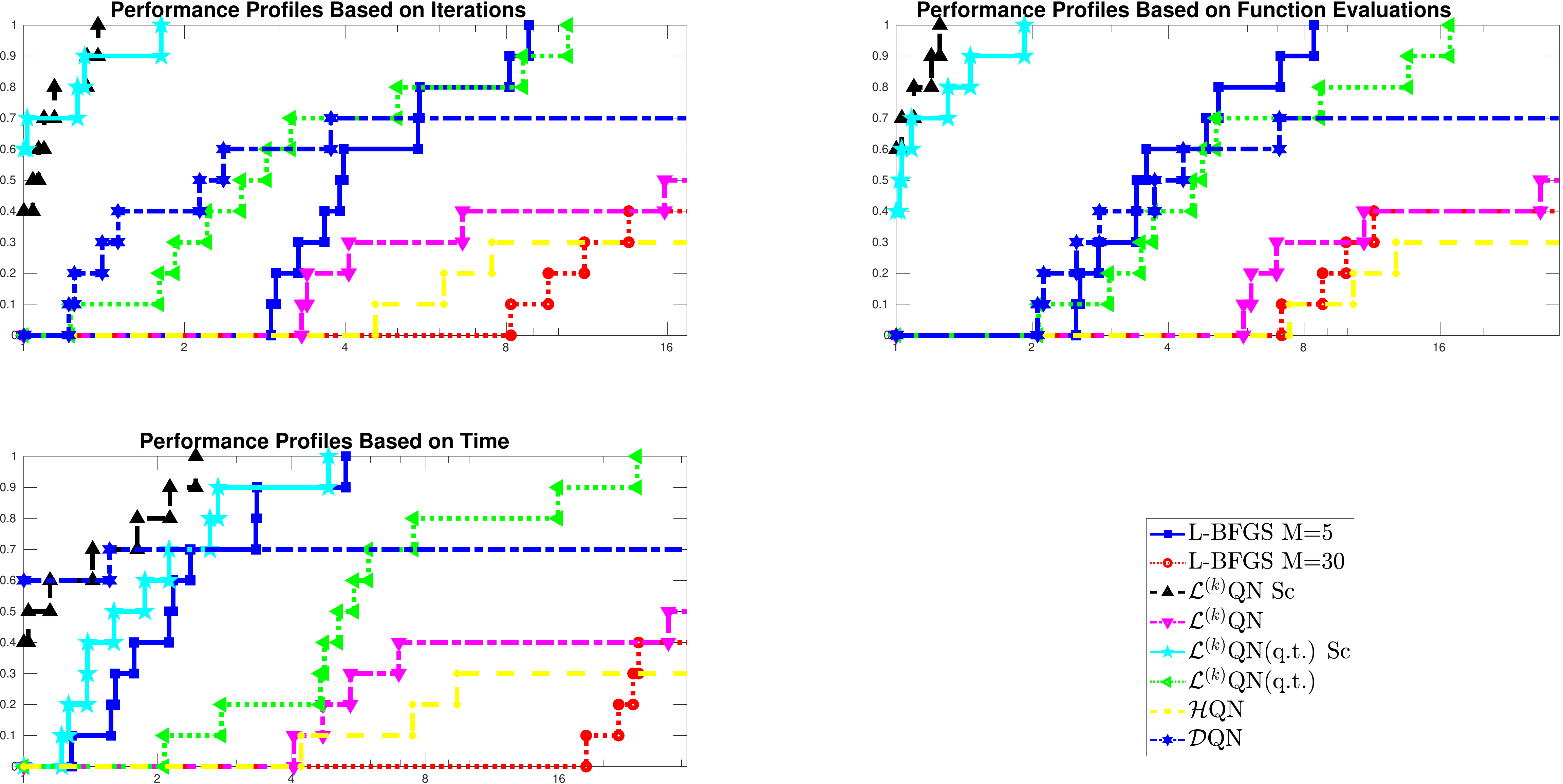}}
		\caption{Performance profiles for Algorithm \ref{convergentlkqn_no_fq}, Algorithm \ref{convergentlkqn}, $\mathcal{D}QN$ \cite{CCD}, $\mathcal{H}QN$ \cite{BDCFZ,DFLZ} and $L$-$BFGS$ with $M=5$ and $M=30$ \cite{LN} when $k=2^7$. LineSearch\_ftol=1e-4; LineSearch\_gtol=0.9.}\label{fig:experiment4}
	\end{figure}

	\begin{table}[htb]
		\centering
		\footnotesize
		\caption{MINST factorization: problems dimensions}
		\label{minst_dimensions}
		\begin{tabular}{lllllllllll}
			Class/Rank & $0$      & $1$      & $2$      & $3$      & $4$      & $5$      & $6$      & $7$      & $8$      & $9$      \\ \hline
			$k=2^6$    & $112896$ & $122816$ & $116224$ & $114816$ & $113024$ & $107264$ & $111488$ & $115968$ & $112512$ & $114752$ \\ \hline
			$k=2^7$    & $225792$ & $245632$ & $232448$ & $229632$ & $226048$ & $214528$ & $222976$ & $231936$ & $225024$ & $229504$\\ \hline
		\end{tabular}
	\end{table}

	\subsection{Conclusions and future works} \label{sec:conclusion_and_future}
	
	In this work we have {proposed}  and studied the convergence of novel optimization scheme{s} $\mathcal{L}^{(k)}$QN obtained by generalizing the updates in the restricted Broyden class by means of projections of the Hessian approximations $B_k$ on adaptive low complexity matrix algebras $\mathcal{L}^{(k)}$, {and in particular, we have studied in detail two new $BFGS$-type methods with theoretical guarantee of convergence}. 
	
	{The finite quadratic termination is not really relevant for general Quasi-Newton methods \cite{kolda1998bfgs}. However, the numerical results presented in the previous subsections, see ``$\mathcal{L}^{(k)}$QN'' and ``$\mathcal{L}^{(k)}$QN(q.t)'' in Figures \ref{fig:experiment1},\ref{fig:experiment3} and \ref{fig:experiment4}, confirm that if this property is added to $BFGS$-type algorithms, as in Algorithm \ref{convergentlkqn}, then we succeed in improving the performances of the basic  $\mathcal{L}^{(k)}$QN scheme in Algorithm \ref{convergentlkqn_no_fq}, { which is a convergent refinement of the methods considered in  \cite{CDTZ}}.}
	
	Moreover the numerical results show that, by an adaptive choice of the matrix algebras $\mathcal{L}^{(k)}$, the robustness of the existing fixed algebras $\mathcal{L}QN$ methods, {$\mathcal{D}QN$  and  $\mathcal{H}QN$,} can be overcome {(see Figures \ref{fig:experiment3} and \ref{fig:experiment4})}, even though this does not guarantee the best performance in terms of Iteration, Function Evaluations or Execution Time (see Figure \ref{fig:experiment1}). Notice, moreover, that the methods $\mathcal{D}$QN and $\mathcal{H}$QN \cite{BDCFZ,CCD,ebrahimib} are competitive for other classes of problems. 
	
	Now, in Experiment 1, the comparison of  the proposed $\mathcal{L}^{(k)}$QN methods is not totally favorable. In fact, Figure \ref{fig:experiment1} shows that the best performers are {$\mathcal{D}QN$  and  $\mathcal{H}QN$}. However, the improved robustness of our proposal, already traceable in Experiment 1, is further underpinned by Experiment 2, where Algorithm \ref{convergentlkqn} always reaches the required level of accuracy within the maximum number of allowed iterations, whereas $L$-$BFGS$ with $M=5$ and $M=30$ drastically changes its behavior when switching from rank $2^6$ to rank $2^7$. In this experiment, a straightforward implementation of our proposals does not guarantee  to outperform  $L$-$BFGS$ with $M=5$. 
	
	However, on this set of problems, the efficiency of our proposals is dramatically improved by introducing 
	a self-scaling factor as outlined in Remark \ref{remark:scaling}. In this case, see ``$\mathcal{L}^{(k)}$QN Sc'' and ``$\mathcal{L}^{(k)}$QN(q.t) Sc'' in Figure \ref{fig:experiment4}, our proposals clearly outperform $L$-$BFGS$ with $M=5$ and $M=30$. 
	
	It is important to note that our proposals dot not require the choice of a problem dependent parameter as $M$ in $L$-$BFGS$ and that,  in general, if $M$ is big, require less memory to be implemented. 
	
	By the above reasons, further investigation urges in order to understand if the new method could be a valid competitor of $L$-$BFGS$, in particular for those problems where large values of the parameter $M$ must be chosen  in order to guarantee satisfactory performances  (see also \cite{jiang2004preconditioned}) or for those problems where the computation of the gradient is expensive, as those coming from data science or optimal control (see, for example, \cite{bottou2018optimization,cipolla2018fractional}). 
	
	{It is clear that $\mathcal{L}^{(k)}QN$ methods should be also compared with the class of nonlinear conjugate gradient methods. Moreover, it would be important to understand} if the matrices generated by means of our Quasi Newton-type updates could be useful as preconditioners for nonlinear conjugate gradient methods as in \cite{andrea2017novel}. Of course, further 
	investigation should be devoted, in future, in order to understand if
	the Broyden Class-version of Algorithm~\ref{convergentlkqn_no_fq} or Algorithm~\ref{convergentlkqn} can produce better performances for $\phi \in (0,1)$. {Last but not least, it could be interesting to understand  if the results presented in this paper can be extended to the modified $BFGS$ method for non-convex functions as in \cite{li2001modified}. Finally the connections with Quasi-Newton Self-Scaling methods \cite{nocedal1993analysis,al1998global} should be further explored.
		
		\section*{Acknowledgments}
		We would like to thank the referees for their thorough reading of the manuscript, valuable suggestions and for pointing to relevant typos.
		
		\section{{Appendix 1: Householder Matrices}}\label{appendix} The results contained in this section are borrowed from \cite{CDZarnoldipreserver} and we refer the interested reader there for more details.
		\begin{definition}[Householder Orthogonal Matrix] \label{onehouse}
			Given a vector $\p \in \mathbb{R}^n$ define 
			$$\mathcal{H}(\p):=I_n-\frac{2}{\|\p\|^2}\p\p ^T.$$
			Consider two vectors $\v,\, \z\in \mathbb{R}^n$. From direct computation one can check that defining $\p= \v - \frac{\|\v\|}{\|\z\|}\z$ with $\z \neq 0,$ we have 
			$$\mathcal{H}(\p)\v=\frac{\|\v\|}{\|\z\|}\z.$$  
		\end{definition}
		\begin{lemma}[\cite{CDZarnoldipreserver}] \label{lem:householder_fixed_columns}
			Consider $W=[\w_1|\dots|\w_s] \in \mathbb{R}^{n\times s}, V=[\v_1|\dots|\v_s] \in \mathbb{R}^{n\times s}$ {of full rank and} such that $s \leq n$, $W^TW=V^TV$. Then there exist  $\,\h_1, \dots,\h_{s} \in \mathbb{R}^n$, {$\|\h_i\|=\sqrt{2}$}, such that the orthogonal matrix $U=\mathcal{H}(\h_s)\cdots\mathcal{H}(\h_1)$, product of $s$ Householder matrices, satisfies the following identities  $$U \w_{i}= \v_i \hbox{ for all } i \in \{1, \dots, s\}.$$
			{The vectors $\h_i$ for $i\in\{1,\dots,s\}$ can be obtained by setting:
				\begin{equation} \label{eq:hexpression}
				\begin{split}
				& { \tilde{\h}_i :=  \mathcal{H}(\h_{i-1}) \cdots \mathcal{H}(\h_1) (\w_{i}-\w_{i-1}) - (\v_{i}-\v_{i-1}),}\\
				& \h_i:= ( \sqrt{2}/\|\tilde{\h}_i\| )\tilde{\h}_i
				\end{split}
				\end{equation}
				(where we set $\mathbf{h}_0=\mathbf{w}_0=\mathbf{v}_0=\boldsymbol{ 0}$). } {If $s=n$ we have $\h_n=\mathbf{0}$ or $\h_n=\frac{\sqrt{2}}{\|\v_n\|}\v_n.$}
			The cost of the computation of the $\h_i$ for $i=1,\dots,s$ is:
			$$[s(s-1)n+ s(2n+1)] \hbox{ mult. }  + [(s(s+2) - 2)n + s(n-1)] \hbox{ add. }  + s \hbox{ sq. roots.}$$
			Observe that when $\w_i=\e_{k_i}$ for $i=1,\dots,s$, that is when $\v_1, \dots, \v_s $ are orthonormal and we are interested to construct an orthogonal $U$ with $s$ columns fixed as $\v_1, \dots, \v_s $, it is possible to save $(s-1)n \hbox{ mult.}$ and $ (3s-2)n \hbox{ add..}$
		\end{lemma}
		\begin{proof}
			The explicit expression of the $\h_i$ in \eqref{eq:hexpression} is obtained by applying the techniques for their construction introduced in \cite{CDZarnoldipreserver}.
		\end{proof}
		
		\section{{Appendix 2: details on Theorem \ref{Sconvergence}  }}
		In order to prove inequality \eqref{eqq7 : nw} it is enough to prove that:
		
		\begin{lemma} \label{lemma:appendix_constant}
			There exists $c_3$ constant with respect to $j$ and depending only on $s$ and $M$ such that 
			\begin{equation*}
			\gamma((j+1-s)+1)^{n} \leq c_3^{j+1-s} \hbox{ for all } j \geq s, \hbox{ where } \gamma:=(\frac{c_1}{n})^n\frac{1}{\det B_s}
			\end{equation*}
			(of course, such $c_3$ turns out to be greater than $1$).
		\end{lemma}
		In fact, once Lemma~\ref{lemma:appendix_constant} is proved, the constant $c_2$ (constant with respect to $j$) for which \eqref{eqq7 : nw} is verified, will be
		$c_2=2c_1c_3/(1-\beta)$ (note that $c_2$ depends only $s$, $M$, $\beta$ but not on $j$).
		\begin{proof}
			Fix $\tilde{c}_3>1$. Note that the sequence of positive numbers
			\begin{equation*}
			\frac{\gamma ((j+1-s)+1)^n}{\tilde{c}_3^{j+1-s}} \hbox{ for } j= s,s+1,\dots
			\end{equation*} 
			converges to zero as $j \to +\infty$; thus there exists $j^{*}\geq s$ (depending on $s$, $M$ and $\tilde{c}_3$) s.t.
			\begin{equation*}
			\gamma ((j+1-s)+1)^n \leq {\tilde{c}_3^{j+1-s}} \hbox{ for all } j \geq j^{*}.
			\end{equation*}
			Note also that for all $j \in \{s+1, \dots , j^{*}-1\}$ we have
			\begin{equation}\label{eq:appendix_gamma} 
			\gamma ((j+1-s)+1)^n \leq \gamma (j^{*}-s+1)^n
			\end{equation}
			and consider $\hat{j} \geq j^{*}$ s.t. $\gamma (\hat{j}-s+1)^n>1$ ($\hat{j}$ depends on $s$, $M$, $\gamma$ and $\tilde{c}_3$). From \eqref{eq:appendix_gamma} we have
			\begin{equation*}
			\gamma ((j+1-s)+1)^n \leq \gamma (\hat{j}-s+1)^n \leq (\gamma (\hat{j}-s+1)^n)^{j+1-s} 
			\end{equation*}
			for all  $j \in \{s,s+1, \dots j^{*}-1\}$.
			
			Collecting the above results, we can conclude that
			\begin{equation}
			\gamma ((j+1-s)+1)^n \leq c_3^{j+1-s} \hbox{ for all } j \geq s
			\end{equation}
			where $c_3:=\max \{ \tilde{c}_3, \gamma(\hat{j}-s+1)^n\}$ ($c_3>1$ and depends on $s$, $M$ and $\tilde{c}_3$).
			
			Finally note that, once $\tilde{c}_3$ is fixed, it is clear that $c_3$ depends only on $s,M$.
		\end{proof}
		In order to prove inequality \eqref{eqq7 : nw_2}, define $a_k:=(1-\phi -\psi_k\phi)\|\g_k\|^2/\s_k^{T}(-\g_k)>0$. We know that $\lim_{k \to +\infty}a_k=+\infty$ and we have to show that there exists $j^{*}\geq s$ such that
		\begin{equation}\label{eqq7 : nw_2_bis} 
		\prod_{k=s}^{j} a_k > c_2^{j+1-s} \hbox{ for all } j \geq j^{*}.
		\end{equation} 
		If $a_k\geq c_2$ for all $k \geq s$, since it must be $a_k>c_2$ for infinite indexes $k$, then the thesis is obvious. So assume that there exists some index $k$ such that $a_k < c_2$. Let $r \geq s$ be such that $a_k>c_2$ for all $k>r$. Note that $c_2>\min_{k=s, \dots, r}a_k$. Set  
		\begin{equation*}
		t:=\big(  \frac{c_2}{\min_{k=s, \dots, r}a_k}  \big)^{r+1-s} >1.
		\end{equation*}
		Let $j^{*}>r+1$ be such that $a_k \geq t c_2$ for all $k\geq j^{*}$. Then we have
		\begin{equation*}
		\begin{split}
		& \prod_{k=s}^{j^{*}}a_k=(\prod_{k=s}^{r}a_k)(\prod_{k=r+1}^{j^{*}-1}a_k)a_{j^{*}}> \\
		& (\min_{k=s,\dots,r}a_k)^{r-s+1}c_2^{j^{*}-r-1}tc_2=\\
		& (\min_{k=s,\dots,r}a_k)^{r-s+1}\big(  \frac{c_2}{\min_{k=s, \dots, r}a_k}  \big)^{r-s+1}c_2^{j^{*}-r}=c_2^{j^*-s+1},
		\end{split}
		\end{equation*}
		i.e., $\prod_{k=s}^{j^{*}}a_k>c_2^{j^*-s+1}$. Thus we obtain \eqref{eqq7 : nw_2_bis} since $a_k\geq tc_2>c_2$ for $k> j^{*}$.
\bibliographystyle{abbrv}
\bibliography{optbib}
\end{document}